\documentclass{amsart}

\usepackage{amsmath}
\usepackage{amsfonts}
\usepackage{amssymb}
\usepackage{amsthm}
\usepackage{comment}
\usepackage{epsfig}
\usepackage{psfrag}
\usepackage{mathrsfs}
\usepackage{amscd}
\usepackage[all]{xy}
\usepackage{rotating}
\usepackage{lscape}
\usepackage{amsbsy}
\usepackage{verbatim}
\usepackage{moreverb}
\usepackage{xspace}
\usepackage{tocvsec2}

\newtheorem{theorem}{Theorem}[section]
\newtheorem{prop}[theorem]{Proposition}
\newtheorem{lemma}[theorem]{Lemma}
\newtheorem{cor}[theorem]{Corollary}

\theoremstyle{definition}
\newtheorem{definition}[theorem]{Definition}

\theoremstyle{remark}
\newtheorem{remark}[theorem]{Remark}
\newtheorem{example}[theorem]{Example}

\newcommand{\Mn}{{\mathcal M}_n}
\newcommand{\Mla}{{\mathcal M}_{\Lambda}(Q)}

\newcommand{\Tnp}{\TT^{n}_p}
\newcommand{\TZ}{\TT^{n}}

\newcommand{\Tp}[1][{1}]{\TT^{#1}_p}
\newcommand{\Cp}{C_{p^\infty}^{\times n}}
\newcommand{\Ver}{Verschiebung\xspace}
\newcommand{\zp}{\bZ_p}
\newcommand{\zpn}{\zp^n}
\newcommand{\zpk}{\zp^k}

\newcommand{\nn}{\nonumber}
\newcommand{\nid}{\noindent}
\newcommand{\ra}{\rightarrow}
\newcommand{\la}{\leftarrow}
\newcommand{\xra}{\xrightarrow}
\newcommand{\xla}{\xleftarrow}

\newcommand{\fib}{\twoheadrightarrow}

\newcommand{\bZ}{\mathbf{Z}}
\newcommand{\bQ}{\mathbf{Q}}
\newcommand{\ess}{\mathbf{S}}
\newcommand{\cC}{\mathcal{C}}
\newcommand{\bC}{\mathbf{C}}

\newcommand{\bR}{\mathbf{R}}
\newcommand{\bN}{\mathbf{N}}
\newcommand{\SpecCat}{\mathrm{Spec}}
\newcommand{\smsh}{\sm}
\newcommand{\we}{\simeq}
\newcommand{\Ar}{\mathcal Ar}
\newcommand{\op}{\mathrm{op}}

\newcommand{\dT}[1][{}]{T_{#1}}
\newcommand{\Tn}{T^{(n)}}

\DeclareMathOperator*{\colim}{colim}
\DeclareMathOperator*{\hocolim}{hocolim}
\DeclareMathOperator*{\holim}{holim}
\DeclareMathOperator{\lcm}{lcm}
\DeclareMathOperator{\bez}{bez}
\DeclareMathOperator{\Cat}{Cat}
\DeclareMathOperator{\Map}{Map}
\DeclareMathOperator{\Aut}{Aut}

\newcommand{\hra}{\hookrightarrow}
\newcommand{\sm}{\wedge}
\newcommand{\id}{\mathrm{id}}

\def\rrarrow{   \hspace{.05cm}\mbox{\,\put(0,-2){$\rightarrow$}\put(0,2){$\rightarrow$}\hspace{.45cm}}}
\newcommand{\rra}{\rrarrow}

\def\cC{\mathcal C}\def\cD{\mathcal D}

\def\cK{\mathcal K}
\def\cM{\mathcal M}\def\cO{\mathcal O}

\def\cU{\mathcal U}

\def\FF{\mathbb F}

\def\NN{\mathbb N}
\def\TT{\mathbb T}

\begin{document}

\author{Gunnar Carlsson} 
\thanks{The first author was supported in part
  by NSF grant DMS-0406992}
\address{Department of Mathematics, Stanford University, Palo Alto, CA 94305, USA}
\email{gunnar@math.stanford.edu}

\author{Christopher L. Douglas} 
\thanks{The second author was supported in part by an NSF Postdoctoral Fellowship.}
\address{Department of Mathematics, Stanford University, Palo Alto, CA 94305, USA}
\email{cdouglas@math.stanford.edu}

\author{Bj{\o}rn Ian Dundas}
\address{Department of Mathematics, University of Bergen, N-5008 Bergen, Norway}
\email{dundas@math.uib.no}

\title{Higher topological cyclic homology and \\ the Segal conjecture for tori}

\begin{abstract}

We investigate higher topological cyclic homology as an approach to studying chromatic phenomena in homotopy theory.  Higher topological cyclic homology is constructed from the fixed points of a version of topological Hochschild homology based on the n-dimensional torus, and we propose it as a computationally tractable cousin of n-fold iterated algebraic K-theory.

The fixed points of toral topological Hochschild homology are related to one another by restriction and Frobenius operators.  We introduce two additional families of operators on fixed points, the Verschiebung, indexed on self-isogenies of the n-torus, and the differentials, indexed on n-vectors.  We give a detailed analysis of the relations among the restriction, Frobenius, Verschiebung, and differentials, producing a higher analog of the structure Hesselholt and Madsen described for 1-dimensional topological cyclic homology.

We calculate two important pieces of higher topological cyclic homology, namely topological restriction homology and topological Frobenius homology, for the sphere spectrum.  The latter computation allows us to establish the Segal conjecture for the torus, which is to say to completely compute the cohomotopy type of the classifying space of the torus.

\end{abstract}

\vspace*{-25pt}
\maketitle

%%%%%%%%%%
%
% 1. Introduction
% 1.1 Background and motivation
% 1.2 Results and future directions
%
% 2. Higher topological cyclic homology
% 2.1 Higher topological Hochschild homology
% 2.2 Restriction and Frobenius operators
% 2.3 Covering homology
% 2.4 Transfer maps
% 2.5 The Verschiebung
%
% 3. Differentials and higher Witt relations
% 3.1 The one-dimensional Witt relations
% 3.2 Higher differentials
% 3.3 Gcd's and Lcm's of matrices
% 3.4 Stable splitting of the transfer
% 3.5 Relations among the Frobenius, differential, and Verschiebung
%
% 4. Adams operations on covering homology
% 4.1 Identification of the Frobenius and restriction categories
% 4.1 Actions on inverse systems
% 4.2 Examples of group actions on covering homology
%
% 5. Calculation of TR^(n) for the sphere
% 5.1 The K-theory of finite G-sets
% 5.2 K-theory and the equivariant sphere spectrum
% 5.3 The homotopy limit over the restrictions
%
% 6. TF^(n) for the sphere and the Segal conjecture for tori
% 6.1 TF^(n) as a homotopy limit of equivariant sphere spectra
% 6.2 The rank filtration of the equivariant sphere spectrum functor
% 6.3 The cotype decomposition of the sphere and a splitting of the rank filtration
% 6.4 The homotopy type of the fixed rank-cotype components of TF^(n)
%
% Appendix. Cyclic homology as a homotopy limit of Frobenius homology
%
% %%%%%%%%%

\setcounter{tocdepth}{2}

\vspace*{-18pt}
\tableofcontents 

\section{Introduction}

\subsection{Background and motivation}

A casual glance at any chart of homotopy groups of spheres is dizzying---one gets the impression that there is some not-quite discernible pattern.  The chromatic viewpoint on stable homotopy theory clarifies matters: we put on colored goggles so that we can see information only of a particular wavelength, that is with particular periodicity properties.  This organizing principle is enormously useful both as a conceptual framework for comprehending large scale phenomena in homotopy theory, and as a computational tool~\cite{ravenel-greenbook, ravenel-orangebook}.

The traditional approach to chromatic phenomena proceeds by what Hopkins calls ``designer homotopy theory".  In this one designs, abstractly and without regard to geometry, spectra with desired chromatic properties: Morava K-theories $K(n)$, Lubin-Tate spectra $E_n$, higher real K-theories $EO_n$, and so on.  This tact is extremely successful: we input the ``wavelength" we are interested in studying, by examining a height $n$ formal group, and build colored goggles to suit this study.  Even the magnificent and quite geometric theory of topological modular forms~\cite{hopkins-icm02} has the feature that the chromatic information is present at the outset---the construction of $tmf$ makes explicit use of the height 1 and 2 formal groups of elliptic curves.

In order to complement this designer homotopy approach to chromatic information, we would like (1) a geometric understanding of the meaning of chromatic phenomena in homotopy theory, and (2) a natural construction of chromatic type n spectra that doesn't begin with the formal group of height $n$, that is in which the chromatic behavior is an output rather than an input to the theory.  Regarding the first desiderata, vector bundles and topological K-theory provide an elegant geometric description of chromatic level 1 structure.  Analogous geometric frameworks for chromatic level 2 structure are beginning to take shape: the Baas-Dundas-Rognes theory of 2-vector bundles~\cite{bdr,bdrr}, and the Segal-Stolz-Teichner theory of 2-dimensional conformal field theories~\cite{segal-cft,stolzteichner} seem particularly promising.  Regarding the second desiderata, Rognes' red-shift conjecture proposes that the algebraic K-theory functor transforms appropriate chromatic type $n-1$ spectra into chromatic type $n$ spectra.  If the conjecture is correct, algebraic K-theory fits the bill of an alchemical way to produce chromatic phenomena.

The red-shift conjecture naturally leads one to investigate the iterated algebraic K-theory $K^n(A)$ of a spectrum $A$, by way of stepping up the chromatic ``spectrum".  In particular, the $n$-fold iterated algebraic K-theory of the ring $\FF_p$ is a natural candidate for a basic type $n$ spectrum.  There are two problems with this iterated algebraic $K$-theory approach.  The first is that it is computationally intractable.  It is difficult enough to compute the effect of a single application of algebraic $K$-theory, much less numerous iterates all at once; indeed already for two-fold algebraic $K$-theory, the computations required considerable effort and ingenuity on the part of, for example, Ausoni and Rognes~\cite{ausonirognes}.  The second problem is that it seems plausible that despite exhibiting chromatic level $n$ behavior, iterated algebraic $K$-theory will not have particularly pleasant universal properties as a chromatic type $n$ spectrum.  For instance, $n$-fold iterated algebraic $K$-theory might have an (as yet technically unavailable) action by the symmetric group---this action and perhaps others would need to be equalized in order to approach a more canonical construction.

The present paper is part of a program to give a simple, direct construction of spectra exhibiting higher chromatic behavior, based not on iterated algebraic K-theory but on forms of higher topological cyclic homology.  The topological cyclic homology $TC(A)$ of a ring spectrum $A$ is built out of fixed point spectra arising from a circle action on the topological Hochschild homology spectrum $THH(A) := A \otimes S^1$.  More specifically, $TC(A)$ is the homotopy limit of the fixed points $(A \otimes S^1)^G$, for finite subgroups $G$ of the circle, over certain restriction and Frobenius operators.  By work of Goodwillie, McCarthy, and the third author~\cite{goodwillie-relalgk, mccarthy-relalgk, dundas-relalgk}, the cyclotomic trace map from algebraic K-theory to topological cyclic homology is a relative equivalence, and so the latter is a reasonable approximation to the former.  Higher topological cyclic homology is built from fixed point spectra associated to the spectrum $THH^n(A) := A \otimes \TT^n$---that is, we replace the circle by the $n$-dimensional torus, with the result being a higher topological Hochschild homology.  In detail, higher topological cyclic homology $TC^{(n)}(A)$ is a homotopy limit of the fixed points $(A \otimes \TT^n)^G$ for finite subgroups $G$ of the torus, over a collection of restriction and Frobenius operators associated to self-isogenies of the torus.

As yet it remains a hope that higher topological cyclic homology is an effective substitute for iterated algebraic K-theory, but in any case we believe it has three advantages: first, it is much more amenable to computation; second, it is both conceptually and technically much simpler; third, it has better symmetry properties, coming from isogenies of the torus that mix the different circle factors, and so we believe it represents a particularly natural candidate red-shift functor.  In order to obtain a clearer picture of the higher topological cyclic homology functor, we need, quite simply, to compute examples.

In order to approach calculational techniques for higher topological cyclic homology, we recall Hesselholt and Madsen's approach to classical topological cyclic homology computations.  In this case the basic operators are the restriction $R_p : THH(A)^{\bZ/p^k} \ra THH(A)^{\bZ/p^{k-1}}$ and Frobenius $F^p : THH(A)^{\bZ/p^k} \ra THH(A)^{\bZ/p^{k-1}}$.  Instead of immediately taking the homotopy limit over the restriction and Frobenius, Hesselholt and Madsen treat the diagram of restriction operators as a prospectrum, and observe that in addition to the Frobenius action, there are two additional operators on this prospectrum, the Verschiebung $V_p: THH(A)^{\bZ/p^{k-1}} \ra THH(A)^{\bZ/p^k}$ and the differential $d_1 : S^1 \sm THH(A)^{\bZ/p^k} \ra THH(A)^{\bZ/p^k}$.  They prove that these operators satisfy the key relations $F^p V_p=p$ and $F^p d_1 V_p =d_1$.  These relations, and a detailed understanding of the formal structure they entail, allow Hesselholt and Madsen to do extensive calculations of topological cyclic homology, and therefore of algebraic K-theory~\cite{hm-kthyfinitealg,hm-localfields, hm-mixedchar}.

One of our main tasks in the present paper is the production and analysis of analogous operators for higher topological cyclic homology.  We already have at our disposal restriction $R_\alpha$ and Frobenius $F^\alpha$ operators~\cite{bcd}, indexed by self-isogenies $\alpha$ of the n-torus.  We introduce higher Verschiebung operators $V_\alpha$, also indexed by isogenies, and differentials $d_v$ indexed by vectors $v \in \bZ^{n}$.  We then establish numerous relations among these operators, including analogues of the classical relations for the products $F^\alpha V_\beta$ and $F^\alpha d_v V_\beta$.

The simplest potential computation is that of the higher topological cyclic homology $TC^{(n)}(\ess)$ of the sphere spectrum.  This spectrum $TC^{(n)}(\ess)$ is the homotopy limit over restriction and Frobenius operators on fixed points of the higher topological Hochschild homology $THH^n(\ess)$ of the sphere.  However, the higher topological Hochschild homology of the sphere is just the sphere spectrum itself, together with a new-found torus equivariance.  We therefore find ourselves investigating fixed points of equivariant sphere spectra, a subject already of considerable classical importance.  The Segal conjecture for a finite group $G$ provides a description of the homotopy type of the homotopy $G$-fixed points $\ess^{h G}$ of the sphere~\cite{carlsson-segalconj}.  Said another way, the conjecture computes the cohomotopy $F(BG_+, \ess) = \ess^{h G}$ of the classifying space of the finite group.

In the process of investigating $TC^{(n)}(\ess)$, we study an important piece of topological cyclic homology, namely topological Frobenius homology $TF^{(n)}(\ess)$, which is the homotopy limit only over the Frobenius maps.  This topological Frobenius homology happens to be homotopy equivalent to the cohomotopy spectrum of the classifying space of the torus.  We give a detailed and complete analysis of the homotopy type of the topological Frobenius homology of the sphere, and therefore establish the Segal conjecture for tori.

\subsection{Results and future directions}

Our first main theorem is the description of the relations among the restriction, Frobenius, Verschiebung, and differential operators on the fixed points of higher topological Hochschild homology.

\begin{theorem}
Fix an odd prime p.  Let $A$ be a connective commutative ring spectrum.  For $\alpha \in M_n(\zp) \cap GL_n(\bQ_p)$ an injective endomorphism of $\zpn$, let $L_{\alpha} := \alpha^{-1} \zpn / \zpn \subset \Tnp$ be a corresponding subgroup of the p-adic n-torus.  Denote by $T^{\alpha} := T_{\Tnp}(A)^{L_{\alpha}}$ the $L_\alpha$-fixed points of the higher topological Hochschild homology of $A$ based on the p-adic n-torus; this is a ring spectrum with multiplication $\mu: T^{\alpha} \sm T^{\alpha} \ra T^{\alpha}$.

There are operators in the stable homotopy category $R_{\alpha}: T^{\beta \alpha} \ra T^{\beta}$ (restriction), $F^{\alpha}: T^{\alpha \beta} \ra T^{\beta}$ (Frobenius), and $V_{\alpha} : T^{\beta} \ra T^{\alpha \beta}$ (Verschiebung).  Moreover, for each p-adic vector $v \in \zpn$, there is an operator $d_v: S^1 \sm T^{\alpha} \ra T^{\alpha}$ (differential).  The restriction and Frobenius are ring maps, the differential is a derivation, and the restriction commutes with the other operators.  These $R$, $F$, $V$, and $d$ maps satisfy the following relations.

\begin{enumerate}

\item[1.] 
$\mu (V_\alpha \sm 1) = V_\alpha \mu (1 \sm F^\alpha)$. 

\item[2.] $F^\alpha V_\beta = |\gcd_{\alpha,\beta}| V_{[\lcm_{\alpha,\beta} / \alpha]} F^{[\lcm_{\alpha,\beta} / \beta]}$.

\item[3.] $d_v F^\alpha = F^\alpha d_{\alpha v}$; \\
$V_\alpha d_v = d_{\alpha v} V_\alpha$.

\item[4.] 
$F^\alpha d_v V_\beta = d_{\bez_\alpha \gcd_{\alpha,\beta}^\dagger v} V_{[\lcm_{\alpha,\beta} / \alpha]} F^{[\lcm_{\alpha,\beta} / \beta]} + V_{[\lcm_{\alpha,\beta} / \alpha]} F^{[\lcm_{\alpha,\beta} / \beta]} d_{\bez_\beta \gcd_{\alpha,\beta}^\dagger v}$.

\end{enumerate}

Here we have chosen matrices $\gcd_{\alpha,\beta}$ and coprime matrices $[\alpha / \gcd_{\alpha,\beta}]$ and $[\beta / \gcd_{\alpha,\beta}]$ such that $\alpha = \gcd_{\alpha, \beta} [\alpha / \gcd_{\alpha,\beta}]$ and $\beta = \gcd_{\alpha, \beta} [\beta / \gcd_{\alpha,\beta}]$.  We have also chosen Bezout matrices $\bez_\alpha$ and $\bez_\beta$ such that $[\alpha / \gcd_{\alpha,\beta}] \bez_\alpha + [\beta / \gcd_{\alpha,\beta}] \bez_\beta = 1$, and coprime matrices $[\lcm_{\alpha,\beta} / \alpha]$ and $[\lcm_{\alpha,\beta} / \beta]$ such that $\alpha [\lcm_{\alpha,\beta} / \alpha] = \beta [\lcm_{\alpha,\beta} / \beta]$.

\end{theorem}

See Theorem~\ref{thm-relations} for a more explicit and complete list of relations.  By taking homotopy groups of the fixed points of higher topological Hochschild homology, we can concisely package the structure of higher topological cyclic homology in terms of operators on a pro multi-differential graded ring.

\begin{cor} \label{cor-pmdgr-intro}
Associated to a connective commutative ring spectrum $A$, there is a pro multi-differential graded ring
$TR^{\alpha}_q(A;p)$ defined as follows.  For each matrix $\alpha \in M_n(\zp) \cap GL_n(\bQ_p)$, there are groups $TR^{\alpha}_q(A;p) := \pi_q(T_{\Tnp}(A)^{L_{\alpha}})$, where $L_\alpha = \alpha^{-1} \zpn / \zpn$.  As $q$ varies these groups form a graded ring.  For each p-adic vector $v \in \zpn$, there is a graded differential $d_v: TR^\alpha_q(A;p) \ra TR^\alpha_{q+1}(A;p)$; these differentials are derivations, are linear  
in the vector $v$, and they graded commute with one another.  The collection $TR^{\alpha}_*(A;p)$ is therefore a multi-differential graded ring.  As $\alpha$ varies these form a pro multi-differential graded ring under the restriction maps $R_\alpha$.

There is a collection of pro-graded-ring operators $F^\alpha: TR^{\alpha \beta}_* \ra TR^\beta_*$, and a collection of pro-graded-module operators $V_\alpha: (F^\alpha)_* TR^{\beta}_* \ra TR^{\alpha \beta}_*$.  That $V_\alpha$ is a module map is equivalent to Frobenius reciprocity: $V_\alpha(x) \cdot y = V_\alpha (x \cdot F^\alpha(y))$.

These operators are subject to the relations
\begin{enumerate}
\item[1.] $F^\alpha V_\beta = |\gcd_{\alpha,\beta}| V_{[\lcm_{\alpha,\beta} / \alpha]} F^{[\lcm_{\alpha,\beta} / \beta]}$.

\item[2.] $d_v F^\alpha = F^\alpha d_{\alpha v}$; $V_\alpha d_v = d_{\alpha v} V_\alpha$.

\item[3.] $F^\alpha d_v V_\beta = d_{\bez_\alpha \gcd_{\alpha,\beta}^\dagger v} V_{[\lcm_{\alpha,\beta} / \alpha]} F^{[\lcm_{\alpha,\beta} / \beta]} + V_{[\lcm_{\alpha,\beta} / \alpha]} F^{[\lcm_{\alpha,\beta} / \beta]} d_{\bez_\beta \gcd_{\alpha,\beta}^\dagger v}$.
\end{enumerate} 
\end{cor}

We undertake computations involving the sphere spectrum.

\begin{prop}
The topological restriction homology of the sphere, that is the homotopy limit of the fixed points $T_{\Tnp}(\ess)^{L_\alpha}$ along the restriction operators, is
$$TR^{(n)}(\ess)\simeq\prod_{\mathcal O\subseteq \zpn} B(\zpn/\mathcal O)_+$$
Here the product varies over the open subgroups $\mathcal
O\subseteq\zpn$.
\end{prop}

Our second main theorem is the Segal conjecture for tori, which amounts to a computation of the topological Frobenius homology of the sphere.

\begin{theorem}
The $p$-adic cohomotopy of the classifying space of the torus is homotopy equivalent to the $p$-completion of topological Frobenius homology,
$$F(B\TZ_+, \ess_p) \simeq TF^{(n)}(\ess)_p$$
and the homotopy groups of $TF^{(n)}(\ess)_p$ are as follows:
$$\pi_*(TF^{(n)}(\ess)_p) = \prod_{k,c} \lim_l \left( \bZ[GL_n(\zp) / \Gamma_{l,k,c}] \otimes \pi_*(\Sigma^{\infty} S^k \sm B \TT^k_+)/p^l \right)$$
Here the product is over $1 \leq k \leq n$ and $c$ is a collection of unordered positive integers $\{n_1, \ldots, n_k\}$.  The limit is over $l \in \bN$, and the group $\Gamma_{l,k,c} \subset GL_n(\zp)$ is the stabilizer of a chosen subgroup $K$ of $C_{p^l}^{\times n}$ of rank $k$ and cotype $c$.
\end{theorem}

\nid See Theorem~\ref{thm-segalconj} for a more complete description of the pieces of this decomposition.  

We briefly mention some of the directions in which we intend to take this project.  The most important next step is explicit computations of the higher topological cyclic homology $TC^{(n)}(\FF_p)$ of the Eilenberg-MacLane spectrum for $\FF_p$.  In fact, this story is a bit more subtle than we let on in the above background discussion, because the most interesting versions of cyclic homology are likely to come not from equalizing all restriction and Frobenius operators (as in $TC^{(n)}$), but in carefully choosing subdiagrams of operators for the homotopy limit.  We touch on the kinds of choices we have in mind in section~\ref{sec-adamsops} below, and describe another important aspect of these sub-homotopy limits: there are systems of Adams operations that survive to act on the homotopy limits, and it seems likely that these operations will play a central role in studying the resulting versions of higher topological cyclic homology.  We reserve, though, a more detailed discussion for another occasion.  

The computations of higher topological cyclic homology can, in part, proceed along lines analogous to the work of Hesselholt and Madsen, namely using a collection of interrelated Tate and spectrum cohomology spectral sequences, together with iterated applications of the norm-cofibration sequence.
Along the way in these computations, it will be convenient to formalize the structure seen in Corollary~\ref{cor-pmdgr-intro} into a notion of Burnside-Witt complex.  Such a complex will have the given relations and will also have a compatible map from the Burnside-Witt vectors for $\zpn$.  It will be particularly worthwhile to investigate an initial such complex, a ``de Rham-Burnside-Witt complex", and to describe the analog of the Burnside-Witt structure for log-rings.

We remark on one last natural continuation of this work, namely that we hope our analysis of the cohomotopy of the classifying space of a torus, viewed as the maximal torus of a compact Lie group, can be used as a bootstrap to establish the Segal conjecture for all compact Lie groups.

\section{Higher topological cyclic homology}

\subsection{Higher topological Hochschild homology}

Let $A$ be a connective commutative $\ess$-algebra.  The topological Hochschild homology of $A$, denoted $THH(A)$, is the spectrum $A \otimes S^1$; here $\otimes$ is the tensor in the category of commutative $\ess$-algebras.  More concretely, $THH(A)$ is the realization of the simplicial spectrum whose $k$-th level is $A^{\sm (S^1)_k}$, where $(S^1).$ is the standard simplicial circle.  In this sense, $THH(A)$ is the ``$S^1$-fold smash power" of $A$.

We are interested not only in the homotopy type but in the equivariant homotopy type of topological Hochschild homology.  In particular, we will be interested in a large collection of operations on the fixed points of topological Hochschild homology, including the restrictions, Frobenii, Verschiebung, and differentials.  The restriction maps only exist on equivariant spectra that have a delicate property called cyclotomicity.  Unfortunately, $THH(A)$ is not cyclotomic and so is not suitable for a detailed investigation of fixed point structures.  This trouble can be rectified by constructing, as in Hesselholt-Madsen~\cite{hm-kthyfinitealg}, a cyclotomic spectrum $T(A)$ that is non-equivariantly homotopy equivalent to $THH(A)$.

Higher topological Hochschild homology $THH^n(A)$ is by definition the spectrum $A \otimes \TT^n$, where $\TT^n$ is the n-torus.  This spectrum can again be given very concretely as the realization of a simplicial spectrum $A^{\sm (\TT^n).}$.  As in the one-dimensional case, this spectrum is equivariantly misbehaved and requires rectification.  One rectification is the Loday construction~\cite{bcd}, which produces a spectrum $\Lambda_{\TT^n}(A)$ that is non-equivariantly equivalent to $THH^n(A)$, but which exhibits the desired higher analogues of cyclotomicity.  Moreover, in the one-dimensional case, $\Lambda_{S^1}(A)$ recovers the equivariant homotopy type $T(A)$.  The notation $\Lambda_X(A)$ is meant to suggest both the ``$X$-fold smash power of $A$", that is $\bigwedge_X A$, and also the origin of the technical aspects of the construction in Loday's work.  However, the reader who is used to the Hesselholt-Madsen notation $T(A)$ might be advised to think of $\Lambda_X(A)$ as ``$T_X(A)$", an equivariant topological Hochschild homology based on $X$ rather than on the circle.

The Loday construction provides a functor $X\mapsto \Lambda_X(A)$ from spaces to
spectra, and hence each $\Lambda_X(A)$ comes with an action of the entire space of endomorphisms of $X$.  This action leads to a particularly rich structure when $X$ is a group $G$, for example the n-torus $\TT^n$.  To isogenies of $G$ (that is surjective homomorphisms with finite kernel) one can associate so-called restriction and Frobenius maps relating the fixed point spectra $\Lambda_G(A)^H$, for varying subgroups $H$ of $G$.  The homotopy limit over these restriction and Frobenius maps is a kind of ``topological $G$-cyclic homology" of $A$.

Instead of considering all isogenies of $G$, we can focus attention on a particular class $P$ of isogenies  and consider only the restriction and Frobenius maps coming from these isogenies.  The homotopy limit over these maps is called the $P$-covering homology of $A$ and is denoted $TC^P(A)$.  These versions of $G$-cyclic homology will be described in more detail in section~\ref{sec-covhom} and particular examples of interest discussed in section~\ref{sec-exampleactions}.

When the group $G$ is the p-complete circle, and the class $P$ is the finite orientation-preserving self-coverings of the circle, the $P$-covering homology is the ordinary topological cyclic homology $TC(A)$ of B{\"o}kstedt, Hsiang, and Madsen~\cite{bhm}.  Notice that this class $P$ is generated by the standard p-fold cover of the circle, and so the only operators in the homotopy limit are the traditional restriction $R_p$ (called $\Phi$ in~\cite{bhm}) and Frobenius $F^p$ (called $D$ in~\cite{bhm}).

In the next section, we describe the restriction and Frobenius operators on the fixed points $\Lambda_G(A)^H$ for a general group $G$, and discuss the specialization to our preferred group, the p-adic n-torus.

\subsection{Restriction and Frobenius operators}

As before, consider a connective commutative $\ess$-algebra $A$, and a group $G$.  The restriction and Frobenius are maps between fixed points of higher topological Hochschild homology $\Lambda_G(A)$:
\begin{align}
& R^H_I : \Lambda_G(A)^H \ra \Lambda_{G/I}(A)^{H/I} \nn \\
& F ^H_I : \Lambda_G(A)^H \ra \Lambda_G(A)^I \nn
\end{align}
Here $H$ is a finite subgroup of $G$, and $I$ is a normal subgroup of $H$.  The Frobenius has a concise description as the inclusion of fixed points.  The restriction by contrast is a more involved operation and we only brush against its definition---technical details can be found in~\cite{hm-kthyfinitealg} and~\cite{bcd}.  Roughly speaking, the spectrum $\Lambda_G(A)^H$ can be described in terms of mapping spaces {\em out of} the union $\bigcup_{K \subset H} S^K$ of fixed points by subgroups of $H$; here the $S$ are certain finite $H$-sets.  When we restrict the source of these mapping spaces to the union $\bigcup_{I \subset K \subset H} S^K$ of fixed point by subgroups containing $I$, the resulting spectrum is $\Lambda_{G/I}(A)^{H/I}$.

The restriction map has inconveniently landed in the fixed points of the higher topological Hochschild homology based no longer on $G$ but on $G/I$.  We solve this problem by considering only situations where $G/I$ is itself isomorphic to $G$.  In particular, this is true provided $I$ is the kernel of a surjective homomorphism $a: G \ra G$ with finite kernel.  We identify $G/I$ with $G$ by the natural isomorphism, and thereby $H/I$ with a subgroup of $G$.  The restriction then maps between fixed points of the single spectrum $\Lambda_{G}(A)$.

The whole situation is more conveniently and directly indexed in terms of the self-homomorphisms of $G$, as follows.  Let $a,b: G \ra G$ be surjective homomorphisms with finite kernels.  Let $L_a$ denote the kernel of $a$.  Denote by $\phi_a: G/L_a \ra G$ the natural isomorphism, and let $\phi_a^* G$ be $G$ considered as a $G$-space with the $G$-action $G \times G \xra{a \times 1} G \times G \xra{\mu} G$.  Note that $\phi_a$ gives a $G$-isomorphism $\phi_a: G/L_a \xra{\cong} \phi_a^* G$.  Throughout the paper we will, perhaps unfortunately, refer to both $\phi_a: G/L_a \ra G$ and $\phi_a^{-1}: G \ra G/L_a$ as $\phi_a$, and also, when convenient for notational reasons, as $\phi^a$.  This isomorphism $\phi_a$ gives us control over the $G$-structures on the various quotients $G/L_a$; we need a similarly careful view of the equivariance of the fixed point spectra $\Lambda_G(A)^{L_a}$.  A priori, the spectrum $\Lambda_G(A)^{L_a}$ is a module over $\ess[G/L_a]=\Sigma^{\infty}(G/L_a)_+$, but we can give it an $\ess[G]$-module structure by the composite
$$
\begin{CD}
  \ess[G]\smsh \Lambda_{G}(A)^{L_a}@>{\phi_a \smsh 1}>>\ess[G/L_a]\smsh \Lambda_{G}(A)^{L_a}@>{\mu}>> \Lambda_{G}(A)^{L_a}
\end{CD}
$$
\nid We will often abbreviate this composite action itself by $\mu$, as we think of it as the unambiguous natural action of $G$ on $\Lambda_G(A)^{L_a}$.

In the description of the restriction map above, replace $H$ by the kernel $L_{ba}$ and replace the subgroup $I$ by $L_a$.  Note that the quotient $H/I = L_{ba}/L_a$ is isomorphic to $L_b$.  With the given $\ess[G]$-module structure, the restriction is now an $\ess[G]$-module map:
$$R_a := R^{L_{ba}}_{L_a} : \Lambda_G(A)^{L_{ba}} \ra \Lambda_G(A)^{L_b}$$
This is the form in which we will use the restriction map throughout this paper.  In these terms, the Frobenius is a map
$$F^b := F^{L_{ba}}_{L_a} : \Lambda_G(A)^{L_{ba}} \ra \Lambda_G(A)^{L_a}$$
The relation of the restriction and Frobenius~\cite{bcd} is pleasantly straightforward:
$R_a F^b = F^b R_a$.

We briefly specialize the above discussion of restrictions and Frobenii to our case of special interest, namely when the group is the p-adic n-torus $\Tnp := \bR^n_p / \bZ^n_p$.  The relevant subgroups $L_{\alpha} \subset \Tnp$ arise as kernels of isogenies (surjective finite-kernel homomorphisms) $\alpha: \Tnp \ra \Tnp$.
The monoid of isogenies of the torus is isomorphic to the monoid $\mathcal{M}_n := M_n(\zp) \cap GL_n(\bQ_p)$ of injective linear endomorphisms of $\zpn$;
the correspondence takes a matrix $\alpha$ over $\zp$ to the corresponding covering $\alpha/\zpn: \bR^n_p/\zpn \ra \bR^n_p/\zpn$.  The kernel of this map $\alpha: \Tnp \ra \Tnp$ is $L_{\alpha} = \alpha^{-1}\zpn/\zpn\subseteq \bQ^{n}_p/\zpn=\Cp\subseteq\Tnp$.  For convenience we  now abbreviate the Loday construction of topological Hochschild homology as $T^{(n)} := \Lambda_{\Tnp}(A)$ and the fixed point spectra as $T^{\alpha} := \Lambda_{\Tnp}(A)^{L_{\alpha}}$.  Note well that the fixed point spectrum only depends on the kernel $L_{\alpha}$ and not on the covering map $\alpha$---however, as a $\Tnp$-equivariant spectrum, with $\Tnp$-action $\mu(\phi_{\alpha}^{-1} \sm 1)$ described above, $T^{\alpha}$ does depend on the particular covering, and this justifies the notation.  In this context we think of the isomorphism $\phi_{\alpha}: \Tnp / L_{\alpha} \ra \Tnp$ as ``multiplication by $\alpha$", and note that it restricts to isomorphisms $\Cp/L_\alpha\cong \Cp$ and $L_{\beta \alpha} / L_{\alpha} \cong L_{\beta}$.

For a covering $\alpha \in \mathcal{M}_n$, the restriction now has the form of a map (really a family of maps),
$$R_{\alpha} : T^{\beta \alpha} \ra T^{\beta}$$
This map depends, even non-equivariantly, on the particular covering $\alpha$ and so we see again that it is essential that we keep track of the coverings and not just the corresponding subgroups of the torus.  The Frobenius now has the form
$$F^{\alpha} : T^{\alpha \beta} \ra T^{\beta}$$
As it is derived from the inclusion of fixed points, the Frobenius really depends only on the subgroups $L_{\alpha}$.  However, as we are interested in the interaction of the Frobenius, Restriction, and other operators, it is convenient to keep them all indexed in one place, namely on the coverings of the torus.

\subsection{Covering homology}  \label{sec-covhom}

We are interested in homotopy limits over classes of restriction and Frobenius operators---these limits are called covering homology and represent versions of cyclic homology.  The restrictions and Frobenii are both indexed on coverings of the p-adic n-torus, but with opposite variance composition rules.  We need an indexing diagram encoding this mixed-variance doubling of the category of coverings, and we do this bookkeeping using a ``twisted arrow category", described below.

Let $G$ be a group.  Consider a collection of surjective homomorphisms $a\colon G\fib G$, and let $\cC$ be the monoid generated by these.  This monoid $\cC$ can be viewed as a category, also denoted $\cC$, with one object.  Our desired ``doubling" of $\cC$ is encoded as follows.

\begin{definition}
Let $\cC$ be a category.  The {\em twisted arrow category}
$\mathcal Ar_\cC$  of $\cC$ has objects the arrows of $\cC$; a morphism from $d\colon v\to y$
to $b\colon w\to x$ is a diagram
$$
  \begin{CD}
    v@>c>>w\\@V{d}VV@V{b}VV\\ y@<{a}<<x
  \end{CD}
$$
in $\cC$.  That is, there is a morphism from $d$ to $b$ for every equation $d=abc$.   We write $(a^*,c_*)$ for this morphism and note that the composition rule reads $(a_0^*,{c_0}_*)(a_1^*,{c_1}_*)=((a_1a_0)^*,(c_0c_1)_*)$.
\end{definition}

\begin{definition}
Define $Frob_{\cC}$, respectively $Res_{\cC}$, to be the subcategory of the twisted arrow category $\mathcal Ar_\cC$ with all objects, but only the morphisms of the form $a^*:=(a^*, \id)$, respectively $c_*:=(\id,c_*)$.
\end{definition}

The basic structure of covering homology~\cite{bcd} is encoded in a functor
$\mathcal Ar_\cC \ra \SpecCat$.
The homomorphism $a:G \ra G$ maps to the spectrum $\Lambda_G(A)^{L_a}$.  The image of the map $c_*\colon bc\to b$ is the restriction map
$R_c\colon \Lambda_{G}(A)^{L_{bc}}\to \Lambda_{G}(A)^{L_b}$ and the
image of the map $a^*\colon ab\to b$ is the Frobenius
$F^a\colon \Lambda_{G}(A)^{L_{ab}}\to \Lambda_{G}(A)^{L_b}$.

There are three important limits associated to this functor, namely over the restriction subcategory, the Frobenius subcategory, and over the whole twisted arrow category.

\begin{definition}
\begin{align}
& TR^{\cC}(A) := \holim_{a \in Res_{\cC}} \Lambda_G(A)^{\ker a} \nn \\
& TF^{\cC}(A) := \holim_{a \in Frob_{\cC}} \Lambda_G(A)^{\ker a} \nn \\
& TC^{\cC}(A) := \holim_{a \in \mathcal Ar_\cC} \Lambda_G(A)^{\ker a} \nn
\end{align}
\end{definition}

\nid The last of these is called the $\cC$-covering homology of A.  

We often restrict attention to coverings of the p-adic n-torus, which as before are encoded in the monoid $\cM_n$ of injective endomorphisms of $\zpn$.  In this case we abbreviate $TR$, $TF$, and $TC$ as follows:
\begin{align}
& TR^{(n)}(A) := TR^{\mathcal{M}_n}(A) = \holim_{\alpha \in Res_{\mathcal{M}_n}} T^{\alpha}(A) \nn \\
& TF^{(n)}(A) := TF^{\mathcal{M}_n}(A) = \holim_{\alpha \in Frob_{\mathcal{M}_n}} T^{\alpha}(A) \nn \\
& TC^{(n)}(A) := TC^{\mathcal{M}_n}(A) = \holim_{\alpha \in Ar_{\mathcal{M}_n}} T^{\alpha}(A) \nn
\end{align}

Note that $TC^{(1)}(A)$ is not precisely ordinary topological cyclic homology, because $TC^{(1)}$ takes into account, in addition to the usual p-fold coverings, orientation reversing self-coverings of the circle.  Nevertheless we refer to $TC^{(n)}(A)$ as the higher topological cyclic homology of $A$.

\subsection{Transfer maps}

The restriction and Frobenius operators by no means exhaust the structure of the fixed points of topological Hochschild homology.  Our next stop is the Verschiebung operator, which is derived from the transfer maps associated to the projections $G/L_b \ra G/L_{ab}$.

Assume $G$ is a compact Lie group (p-adic or otherwise), and consider as before a monoid $\cC$ of linear self-coverings of $G$, that is of surjective homomorphisms $a\colon G\to G$ with finite kernel $L_a$.  For $a,b\in\cC$, choose a $G$-representation $W$ and an open $G$-embedding $i\colon W\times G/L_b \hra W\times G/L_{ab}$ over the projection $pr_a\colon G/L_b\to G/L_{ab}$.  The one-point compactification of this embedding, that is the Thom construction, is a $G$-map 
$$tr_{L_a}\colon S^W\smsh (G/L_{ab})_+\to S^W\smsh (G/L_b)_+$$ 
called the $L_a$-transfer.  The transfer is independent of the choice of embedding in the following sense.  Choose a complete universe of $G$-representations, and for any $G$-space $X$, let $Q_G(X)=\colim_W \Map_*(S^W,S^W\smsh X)$ where the colimit is taken over the universe.  Then two different choices of embeddings give homotopic transfers 
$tr_a\colon Q_G(G/L_{ab})\to Q_G(G/L_b).$ %%%

For instance, if $G$ is the circle and $a$ is multiplication by the positive integer $n$ (or rather the $n$-th power operation, as we are writing $G$ multiplicatively), we have an embedding 
$\bC\times G \hra \bC\times \phi^*_nG  \xra{1 \times \phi_n} \bC\times G/L_n$
where $\bC$ is the complex plane with the usual circle action---the embedding can be given explicitly by, for example, $(w,z)\mapsto \left(nz+\frac{1}{1+|w|}w,z^n\right)$.  As a result we have the desired transfer
$tr_n\colon S^{\bC} \smsh (G/L_n)_+ \to S^{\bC} \smsh G_+.$

The properties of the transfer we will need are the following.  (Here $\simeq$ means
``homotopic after applying $Q_G$'', and similarly for commutativity claims.)  

\begin{prop} \label{prop-transfer}
For $a,b:G \ra G$ surjective finite-kernel homomorphisms of compact Lie groups, the transfers $tr_a$ of the projections $pr_a: G \ra G/L_a$ and $pr_a: G/L_b \ra G/L_{ab}$ and the transfers $tr_b$ of the projections $pr_b: G \ra G/L_b$ and $pr_b: G/L_a \ra G/L_{ba}$ satisfy the following relations.

\begin{enumerate}
\item[1.] $tr_a tr_b\simeq tr_{ba}$.  If $f \colon G \to G'$ is an isomorphism of Lie groups and $a'=faf^{-1}$, then $tr_{a'} \simeq f \, tr_a f^{-1}$.  If $a$ is an isomorphism, then $tr_a \simeq \id$.  
\item[2.] \label{trmodule} The transfer is a $G$-module map; that is the following diagram commutes:
$$
  \begin{CD}
    G_+\smsh (S^W\smsh (G/L_b)_+)@<{1\smsh tr_a}<< G_+\smsh (S^W\smsh (G/L_{ab})_+)\\
    @V{\mu}VV@V{\mu}VV\\
    S^W\smsh (G/L_b)_+@<{tr_a}<< S^W\smsh (G/L_{ab})_+
  \end{CD}
$$
Here $\mu$ abbreviates the composite $G_+ \sm (G/L_b)_+ \xra{\phi_b \sm 1} (G/L_b)_+ \sm (G/L_b)_+ \xra{\mu} (G/L_b)_+$.
\item[3.] \label{trcomodule} \emph{Frobenius reciprocity.}  The transfer is a comodule map: the diagram
$$
  \xymatrix{
  S^W\smsh (G/L_b)_+\smsh (G/L_{ab})_+&
  S^W\smsh (G/L_{ab})_+ \smsh (G/L_{ab})_+ \ar[l]_{tr_a\smsh 1}\\
   S^W\smsh (G/L_b)_+ \smsh (G/L_b)_+\ar[u]^{1\smsh (pr_a)_+}&\\
    S^W\smsh (G/L_b)_+\ar[u]^{1\smsh \Delta_+}&
     S^W\smsh (G/L_{ab})_+\ar[l]_{tr_a}\ar[uu]^{1\smsh\Delta_+}
  }
$$
commutes, where $\Delta$ is the diagonal.
\item[4.] \label{trdoublecoset} \emph{The double coset formula.}
  Assume $G$ is connected, and consider a commuting diagram of self-coverings
\vspace{-5pt}
$$
  \begin{CD}
    G@<{\tilde b}<<G\\@V{a}VV@V{\tilde a}VV\\ G@<{b}<< G
  \end{CD}
$$
with $L_{\tilde a}\cap L_{\tilde b}=\{1\}$.
Then 
$$
\begin{CD}
  G/L_{\tilde b}@<{pr_{\tilde b}}<<\coprod_{L_{\tilde b}\backslash L_{a \tilde b}/L_{\tilde a}}G
\\
  @V{pr_a}VV @V{pr_{\tilde a}}VV\\
  G/L_{a \tilde b}=G/L_{b \tilde a}@<{pr_b}<<G/L_{\tilde a}
\end{CD}
$$
is cartesian and 
$$tr_b pr_a \simeq |L_{\tilde b}\backslash L_{a \tilde b}/L_{\tilde a}|\cdot pr_{\tilde a}tr_{\tilde b}.$$
\end{enumerate}
\end{prop}
\nid These properties are standard.  The second and third properties follow from the naturality of the transfer, which along with the first composition property appears already in Kahn and Priddy's original paper~\cite{kahnpriddy-transfer}; the double coset formula predates even the definition of the transfer.
Lest these formulas seem askew, the reader should keep in mind that, despite the notation, $tr_a$ is the transfer for the projection map $pr_a$ and not for the map $a$ itself.

%%%%

\subsection{The Verschiebung} \label{sec-vers}

Our first new operator on the fixed points of topological Hochschild homology, the Verschiebung, is induced by transfer maps for projections associated to coverings.

As before let $G$ be a compact Lie group, $W$ a real $G$-representation, and $a$ a finite self-covering of $G$.  Recall that the Loday construction $\Lambda_G(A)$ is, a priori, a $\Gamma$-space, and so associates to a finite set a space, or more generally to a simplicial set a simplicial space, therefore by realization a space.  (Henceforth we generally do not distinguish between spaces and their singular complexes or between simplicial spaces and their realizations.)  An immediate consequence of the proof of the fundamental cofibration sequence~\cite{bcd} is  that the stabilization map
\begin{equation*}
  \Lambda_{G}(A)(K)^{L_a}\to \Map_*(S^W, \Lambda_{G}(A)(S^W\smsh K))^{L_a} \cong 
  \Map_*(S^W\smsh (G/L_a)_+, \Lambda_{G}(A)(S^W\smsh K))^{G}
\end{equation*}
is a weak equivalence.  Here $K$ is a test input to the $\Gamma$-space $\Lambda_G(A)$.  (In proving this equivalence, the fundamental cofibration sequence provides the induction step with respect to the order of $L_a$.  Alternatively one can approach this stabilization problem via equivariant obstruction theory as in B{\"o}kstedt-Hsiang-Madsen~\cite{bhm}).

Hence the transfer yields a map
$$
\begin{CD}
  \Lambda_{G}(A)(K)^{L_b}@>{\sim}>>  
  \Map_*(S^W\smsh (G/L_b)_+, \Lambda_{G}(A)(S^W\smsh K))^{G}\\
  @.@V{tr_a^*}VV\\
  \Lambda_{G}(A)(K)^{L_{ab}}@>{\sim}>>  
  \Map_*(S^W\smsh (G/L_{ab})_+, \Lambda_{G}(A)(S^W\smsh K))^{G}
\end{CD}
$$
As $K$ varies these fit together into a map in the stable homotopy category; this map is the {\em Verschiebung} 
$$V_a\colon \Lambda_{G}(A)^{L_b}\to \Lambda_{G}(A)^{L_{ab}}.$$
By construction the Verschiebung commutes with restriction operators, that is $R_aV_b=V_bR_a$, and is variously related to the Frobenius as follows:
\begin{prop}\label{prop-vers}
  Let $A$ be a commutative connective $\ess$-algebra, $G$ a compact
  connected Lie group, and $a$ and $b$ finite self-coverings of $G$.  Let $T^a:=\Lambda_{G}(A)^{L_a}$.
  \begin{enumerate}

  \item[1.] \label{FmultVmod} \emph{Frobenius reciprocity.}  $V_b$ is a module map in the sense that the following diagram commutes:
    $$
    \xymatrix{T^b\smsh T^{ab}\ar[r]^{V_a\smsh 1}\ar[d]^{1\smsh F^a}&
      T^{ab}\smsh T^{ab}\ar[dd]^{\mu}\\
      T^b\smsh T^b\ar[d]^{\mu}&\\
      T^b\ar[r]^{V_a}&T^{ab}
      }
    $$
    In particular $V_aF^a=V_a(1)\cdot$.  
  
  \item[2.] \emph{The double coset formula.}  $F^aV_b = |L_{\tilde b}\backslash L_{a \tilde b}/L_{\tilde a}|\cdot V_{\tilde b}F^{\tilde a}$ where $\tilde a$ and $\tilde b$ are self-coverings such that $a \tilde b=b \tilde a$ and $L_{\tilde a} \cap L_{\tilde b} = 1$.

  \item[3.] \label{Famu}
$F^a\mu(a_+\smsh 1) = \mu(1\smsh F^a)\colon G_+\smsh T^{ab}\to T^{b}$ and $V_a\mu = \mu(a_+\smsh V_a)\colon G_+\smsh T^b\to T^{ab}$.

  \item[4.] If 
\vspace{-10pt}
$$
    \begin{CD}
      G@<{\tilde b}<<G\\
      @V{a}VV@V{\tilde a}VV\\
      G@<{b}<<G
    \end{CD}
$$ 
is a cartesian square of finite self-coverings, then
$$F^a\mu(1\smsh V_b)\gamma = \mu(1\smsh F^aV_b)+F^aV_b\mu$$ 
where $\gamma\colon G_+\smsh T^{{\tilde a}}\coprod_{G_+\smsh T^{{\tilde a}}}G_+\smsh T^{{\tilde a}}\to G_+\smsh T^{{\tilde a}}$ is the map induced by $(a_+\smsh 1)+(b_+\smsh 1)$ and the maps in the pushout are $\tilde b_+\smsh 1$ and $\tilde a_+\smsh 1$.  Otherwise said, the diagram
$$\xymatrix{&G_+\smsh T^{\tilde a}\ar[rd]_{a_+\smsh 1}\ar[rrrd]^{\mu(1\smsh F^aV_b)}&&&\\
G_+\smsh T^{\tilde a}\ar[ur]^{\tilde b_+\smsh 1}\ar[dr]_{\tilde a_+\smsh 1}&&G_+\smsh T^{\tilde a}\ar[rr]^{F^a\mu(1\smsh V_b)}&&T^{\tilde b}\\
&G_+\smsh T^{\tilde a}\ar[ur]^{b_+\smsh 1}\ar[rrru]_{F^aV_b\mu}&&&
}
$$
commutes up to homotopy.

   \end{enumerate}
\end{prop}

\nid Frobenius reciprocity and the double coset formula follow from the corresponding properties of the transfer.  Property 3 simply records the expected equivariance properties of the Frobenius and Verschiebung, and property 4 is a combination of the two parts of property 3. 

%%%%%%%%%%

\section{Differentials and higher Witt relations}

\subsection{The one-dimensional Witt relations}

The fixed points $T(A)^{C_{p^k}}$ of ordinary topological Hochschild homology are related by the operators restriction $R: T(A)^{C_{p^k}} \ra T(A)^{C_{p^{k-1}}}$ and Frobenius $F: T(A)^{C_{p^k}} \ra T(A)^{C_{p^{k-1}}}$.  By definition topological cyclic homology $TC(A)$ is the homotopy limit of the collection $\{T(A)^{C_{p^k}}\}$ over these two series of operators.  We can break this homotopy limit into two stages by considering either only the Restriction or only the Frobenius operators.  The homotopy limit over the restriction maps is called $TR(A)$; topological cyclic homology $TC(A)$ can then be obtained as the homotopy equalizer of $TR(A) \overset{F}{\underset{\id}{\rightrightarrows}} TR(A)$.  Alternately, the homotopy limit over the Frobenius maps is called $TF(A)$; topological cyclic homology can then be recovered as the homotopy equalizer of $TF(A) \overset{R}{\underset{\id}{\rightrightarrows}} TF(A)$.

In order to compute the homotopy groups of topological cyclic homology, it is useful to exploit two additional operations on the homotopy groups of the fixed point spectra.  The first of these we have already discussed, the Verschiebung $V: \pi_*(T(A)^{C_{p^{k-1}}}) \ra \pi_*(T(A)^{C_{p^{k}}})$.  As the Verschiebung arises from a transfer map, it is only well defined at the level of homotopy.  The second operation is a differential $d: \pi_*(T(A)^{C_{p^{k}}}) \ra \pi_{*+1}(T(A)^{C_{p^{k}}})$; the differential is, roughly speaking, multiplication by the fundamental class of $S^1$ on the $S^1$-spectrum $T(A)^{C_{p^k}}$.

The collection ${R, F, V, d}$ satisfies certain fundamental relations: $RF=FR$, $RV=VR$, $Rd=dR$, $FV=p$, $VF=V(1)\cdot$, $FdV=d$.  Moreover $d$ is a differential and a graded derivation.  Hesselholt and Madsen~\cite{hm-mixedchar} proposed viewing the groups $\pi_*(T(A)^{C_{p^k}})$ as a prosystem with respect to the maps $R$, and viewing $F$, $V$, and $d$ as operators on this prosystem.  They formalized this structure in the notion of a Witt complex, and constructed an initial object in the category of Witt complexes.  This universal Witt complex is called the de Rham-Witt complex, in part because it receives a canonical surjective map from the de Rham complex on Witt vectors.  By studying the structure of the de Rham-Witt complexes of rings and also of log-rings, Hesselholt and Madsen were able to do extensive calculations of topological cyclic homology; they applied these calculations with great success to, among other problems, the analysis of the algebraic K-theory of local fields~\cite{hm-localfields}.

In the following sections we begin the investigation of similar structures on higher topological Hochschild homology, focusing on establishing the basic relations among the higher restriction, Frobenius, Verschiebung, and differentials.  Though we do not claim to have determined all possible relations among these operators, we have established higher analogs of all the classical relations---in particular, we describe an intriguing splitting that occurs in the fundamental higher $FdV$ relation.

\subsection{Higher differentials} \label{sec-higherdiff}

We now introduce the differentials in our higher analog of the
Witt structure.  These maps are derived from the
$\Tnp$-action on $\Tn(A)^{L_\alpha}$ by means of
a stable splitting of the torus, and as such are dependent on
the matrix $\alpha$ and not---as the Frobenius and Verschiebung maps are---only on
the group $L_\alpha$.

The transfer map for the projection $\TT^1 \ra *$ is a stable map $S^1 \ra \TT^1_+$ which produces a stable splitting $\TT^1_+ \simeq S^0 \vee S^1$.  In particular the stable homotopy of the torus is $\pi^S_*(\TT^1_+)\cong \pi^S_*(S^0)\oplus\pi_*^S(S^1)$.  The transfer may be given explicitly by the projection $\sigma\colon S^{\mathbf C}=S^2\to S^2/S^1\cong
S^1\smsh \TT^1_+$ sending a point $z\in S^{\mathbf C}$ to 
$\frac{|z|}{1+|z|}\smsh \frac{z}{|z|}\in\mathbf R/\mathbf
Z\smsh\TT^1_+$.  By smashing this map with itself $k$ times, we get a map $\sigma\colon S^{k\mathbf C}\to S^k\smsh\TT^k_+$.  Concretely, this $k$-fold product of the transfer is the Thom construction on the embedding $\bR^k\times\TT^k\hookrightarrow \bC^k$ of a tubular neighborhood of the standard inclusion $\TT^k\subseteq \bC^k$.  Altogether we get a stable splitting
$$\TT^k_+\simeq\bigvee_{T\subseteq\{1,\dots,k\}}S^{|T |}\cong\bigvee_{j=0}^k\left(S^j\right)^{\vee\binom{k}{j}}$$  
We will often write maps between $\TT^{k}_+$'s as matrices in terms
of this last basis.

\begin{lemma}
  With respect to the stable splitting $\TT^n_+\simeq\bigvee_{j=0}^k\left(S^j\right)^{\vee\binom{k}{j}}$ the multiplication $\mu\colon\TT^2_+\to \TT^1_+$ is homotopic to the map
$$\left[
    \begin{matrix}
      1&0&0&0\\0&1&1&\eta
    \end{matrix}
\right]\colon
\begin{matrix}
  S^0\\S^1\\S^1\\S^2
\end{matrix}
\to
\begin{matrix}
  S^0\\S^1
\end{matrix}
$$
\end{lemma}
\begin{proof}
  It is enough to verify that the $S^2\to S^1$ component is given by $\eta$, and this follows from the diagram
$$
  \begin{CD}
    S^{2\bC}@>{\sigma\smsh 1}>>S^1\smsh\TT^1_+\smsh S^\bC@>{1\smsh\mu}>>S^1\smsh S^\bC\\
    @. @V{1\smsh\sigma}VV@V{1\smsh\sigma}VV\\
    @. S^1\smsh\TT^1_+\smsh S^1\smsh\TT^1_+@>{(1\sm \mu) \tau}>>S^1\smsh S^1\smsh\TT^1_+
  \end{CD}
$$
The top $\mu$ refers to the natural action of $\TT^1$ on $S^\bC$.
The square is strictly commutative, and the top horizontal composite is homotopic to $\eta$~\cite{hess-ptypical}.
\end{proof}

\begin{remark}\label{rem-etanonzero}
  For $\alpha\in M_n\zp$, we have an induced stable map $\alpha_+\colon(\Tnp)_+\to(\Tnp)_+$.  From the above description of the multiplication $\mu$ we can describe $\alpha_+$ in terms of the stable splitting.  If $n=1$ and $\alpha=a\in\zp$ then the map $\alpha_+$ is given by 
$\left[
  \begin{smallmatrix}
    1&0\\0&a
  \end{smallmatrix}
\right]$.
  If $n=2$ and 
$\alpha=\left[
    \begin{smallmatrix}
      a&b\\
      c&d
    \end{smallmatrix}
\right]$
then $\alpha_+$ is given by
$$\left[
  \begin{matrix}
    1&0&0&0\\
    0&a&b&ab\eta\\
    0&c&d&cd\eta\\
    0&0&0&ad-bc
  \end{matrix}
\right]\colon
  \begin{matrix}
    S^0\\S^1\\S^1\\S^2
  \end{matrix}
\to
  \begin{matrix}
    S^0\\S^1\\S^1\\S^2
  \end{matrix}
  $$
For general $\alpha=[a_{ij}]\in M_{k\times n}\zp$ the stable map $\alpha_+$ is given by the matrix indexed by the subsets of $\{1,\dots,k\}$ and $\{1,\dots,n\}$ with $S$-$T$-entry 
$$M_{S,T}=\left(
\sum_{f\colon T\fib S}
\text{sign}(f)\prod_{j\in T}a_{f(j)j}
\right)\eta^{|T|-|S|}
\colon S^{|T|}_p\to S^{|S|}_p$$
where the sum is over all surjective $f\colon T\fib S$ and the sign is with respect to the ordering given by the fact that we are considering sets of natural numbers.

If $p$ is an odd prime, multiplication by $\eta$ is nullhomotopic, and these formulae simplify considerably.  In particular we have that if $\alpha\in M_n\zp$, then 
$$
\begin{CD}
  S^{n\bC}_p@>\sigma>> (S^n\smsh{\TT^{n}_{p \; +}})_p\\
@V{\det(\alpha)}VV@V{1\smsh \alpha_+}VV\\
  S^{n\bC}_p@>\sigma>> (S^n\smsh{\TT^{n}_{p \; +}})_p
\end{CD}
$$
commutes up to stable homotopy.
\end{remark}
\begin{definition}
  Let $X$ be a $\Tnp$-spectrum and $\ell\in M_{k\times n}\zp$.
  Define the $\ell^{\text{th}}$ differential as the stable map 
$$d_\ell\colon S^k\smsh X\to X$$
given by the composite
$$
\begin{CD}
S^{k\mathbf C}\smsh X @>{\sigma\smsh 1}>>  S^k\smsh{\TT_p^k}_+\smsh X@>{1\smsh \ell_+\smsh 1}>>
  S^k\smsh{\TT_p^n}_+\smsh X
  @>\mu>>S^k\smsh X
  \end{CD}
$$
\end{definition}
\begin{prop} 
If $\ell'\in M_{n\times k'}\zp$  and $\ell''\in M_{n\times
  k''}\zp$, then 
$$d_{\ell'}d_{\ell''} \simeq d_{\ell}$$ 
where $\ell$ is the $n\times(k'+k'')$ matrix obtained by placing
the columns of $\ell'$ before those of $\ell''$.  

For $p$ an odd prime, $\ell\in M_{n\times k}\zp$,
  and $\gamma\in M_k\zp$, we have 
$$d_{\ell\gamma} \simeq \det(\gamma)\cdot  d_\ell$$
  In particular, if $\ell\in M_n\zp$, then
  $d_\ell \simeq \det(\ell)\cdot d_{1}$.  If $\ell\in M_{n\times k}\zp$ has
  rank less than $k$, then $d_\ell \simeq 0$.
\end{prop}

\begin{proof}
The first statement follows from the commutativity of
$$
\begin{CD}
  \Tp[k']\times\Tp[k'']@>{\ell'\times\ell''}>>\Tnp\times\Tnp\\
  @V{\cong}VV@V{\mu}VV\\
  \Tp[k]@>{\ell}>>\Tnp
\end{CD}
$$
\nid The second statement follows from the observation in remark~\ref{rem-etanonzero} that if $\gamma \in M_n\zp$ and $\eta=0$ then $(1 \sm \gamma_+) \sigma \simeq \sigma \det(\gamma)$.
Finally, if $\ell$ has less than
maximal rank, then $\ell=\ell'\gamma$ where $\gamma\in
M_k\zp$ has zero determinant.
\end{proof}
\nid Note that the equivalence $d_{\ell\gamma} \simeq \det{\gamma}\cdot  d_\ell$ implicitly depends on a $p$-completion; of course, this completion is unnecessary if the matrix $\gamma$ is integral.

\begin{remark}

The first part of this proposition says that all differentials can be described as
composites of ``one-dimensional'' differentials, that is by composites of $d_\ell$'s with
$\ell\colon\zp\to\zpn$.

  For $p=2$ the second part of the proposition becomes a bit more complicated: 
$$d_{\ell\gamma}\simeq
\det(\gamma)\cdot d_\ell+
\sum_{\emptyset\neq
  T\subsetneq\{1,\dots,k\}}M_{\{1,\dots,k\},T}\,d_{\ell i_T}
$$
where $i_T$ is the matrix of the inclusion of $T$ in $\{1,\dots,k\}$, and $M_{S,T}$, as in remark~\ref{rem-etanonzero}, consists of products of minors and powers of $\eta$.  

As an example, let $n=1$ and $\ell=1$.  Then
$d^2_1 \simeq d_{[1,1]}=d_{[1,0]\cdot\left[
    \begin{smallmatrix}
      1&1\\0&0
    \end{smallmatrix}\right]}
    $.
This differential is homotopic to $0\cdot d_{[1,0]}+1\cdot1\cdot d_1\eta+0\cdot0\cdot d_0\eta=d_1\eta$, recovering Hesselholt's formula $d^2\simeq d\eta$ for the one-dimensional case.
\end{remark}

\begin{lemma}
  Let $\ell_1,\ell_2\in M_{n\times 1}\zp$.  Then $d_{\ell_1+\ell_2} \simeq d_{\ell_1}+d_{\ell_2}$.
\end{lemma}
\begin{proof}
The differential $d_{\ell_1 + \ell_2}$ and the sum $d_{\ell_1} + d_{\ell_2}$ are the upper and lower composites of the diagram
$$\xymatrix{
(S^1 \vee S^1) \sm X \ar[r] & (\TT^1_+ \vee \TT^1_+) \sm X \ar[r]^{{\ell_1}_+ \vee {\ell_2}_+} &
(\TT^n_+ \vee \TT^n_+) \sm X \ar[r] \ar[d]^{\text{fold}} & X \vee X \ar[d] \\
S^1 \sm X \ar[r] \ar[u] &
\TT^1_+ \sm X \ar[u]_{\text{pinch}} \ar[r]^{(\ell_1 + \ell_2)_+} & \TT^n_+ \sm X \ar[r] & X
}$$
That $\{\text{fold} \circ ({\ell_1}_+ \vee {\ell_2}_+) \circ \text{pinch}\}: \TT^1_+ \ra \TT^n_+$ is stably homotopic to $(\ell_1 + \ell_2)_+$ is checked by explicit computation in homotopy and depends essentially on the source having only cells in dimensions zero and one.
\end{proof}

%%%%%

\subsection{Gcd's and Lcm's of matrices} \label{sec-gcdlcm}

In order to calculate the relations among the Frobenius, the differential, and the Verschiebung, it is convenient to develop some technology describing how various matrices (which index the $F$, $d$, and $V$ operators) interact.  We do so in a bit of (we hope entertaining) generality.
 
\begin{definition}\label{def:adjoint}
  Let $f\colon A\to B$ be an injection of abelian groups.  The {\em volume}
  $|f|$ of $f$ is defined to be the cardinality of the
  cokernel of $f$. 
The {\em adjoint} $f^\dagger$ of $f$ is the unique
  lifting $f^\dagger$ in 
$$\xymatrix{
0\ar[r]&A\ar[r]^{f}\ar[d]_{|f|}&
B\ar[r]\ar[d]^{|f|}\ar@{-->}[dl]_{f^\dagger}&
B/f(A)\ar[r]\ar[d]^{0}&0\\
0\ar[r]&A\ar[r]^{f}&B\ar[r]&B/f(A)\ar[r]&0
}
$$
\end{definition}

Note that if $A=B=\bZ^n$, then the volume is the absolute value of the
determinant, and the adjoint is given by plus or minus the adjoint of a
matrix in elementary linear algebra.  In $\zp$ there are more units,
and so potentially a greater distance between our adjoint and the
classical adjoint defined by explicit formulae involving minors.

In the following, $\Lambda$ will be a principal ideal domain, and $Q$ a
finitely generated free
$\Lambda$-module.
\begin{definition}
  Define $\Mla\subseteq End_\Lambda(Q)$ to be the submonoid consisting of the injective
  endomorphisms of $Q$.  In the situation $\Lambda=\zp$ and $Q=\zpn$
  we simply write $\Mn:=\mathcal{M}_{\zp}(\zpn)$.  

Given $f,g\in\Mla$, we say
  that $f$ and
  $g$ are {\em coprime} if $f+g\colon Q\oplus Q\to Q$ is surjective.  The {\em
    greatest common divisor} of $f$ and $g$ is a map $\gcd(f,g): Q \ra Q$ for which there are coprime $\bar{f}$ and $\bar{g}$ such that the composite $Q \oplus Q \xra{\bar{f} + \bar{g}} Q \xra{\gcd(f,g)} Q$ is equal to $Q \oplus Q \xra{f + g} Q$; more specifically, the greatest common divisor is the equivalence class in $\Mla/\!\Aut_\Lambda(Q)$ consisting of those $d \in \Mla$ such that $f+g = d(\bar{f}+\bar{g})$ for some coprime $\bar{f}$ and $\bar{g}$.
\end{definition}

Alternately, the greatest common divisor can be viewed as the collection of all maps $Q \ra Q$ that factor into an isomorphism $Q \cong (f+g)(Q \oplus Q)$ followed by the inclusion $(f+g)(Q \oplus Q) \subseteq Q$.  Note that such factorizations exist because $(f+g)(Q \oplus Q)$ is a free module of rank equal to the rank of $Q$---this follows because $f$ is injective and submodules of the free module $Q$ over the PID $\Lambda$ are free.

Note that $|\gcd(f,g)|=1$ and $\gcd(f,g)=\Aut_\Lambda(Q)$
are just other ways of saying that $f$ and $g$ are coprime.

A simple diagram chase gives the following lemma.
\begin{lemma}
  If $f,g\in \Mla$ then 
$\frac{|f||g|}{|Q/(f(Q)\cap g(Q))|}=|\gcd(f,g)|.
$
\end{lemma}

\begin{definition}
  For each coprime pair $f,g\in\Mla$, \emph{choose} a section
$(\bez_f \oplus \bez_g) \colon Q\to Q\oplus Q$ of $f+g$.  The maps $\bez_f$ and $\bez_g$ such that $f \bez_f + g \bez_g = 1$ are the analogs of classical Bezout numbers.  In this coprime case, choose the representative $1$ for the greatest common divisor of $f$ and $g$.  

If $f$ and $g$ are not coprime \emph{choose} a representative in $\Mla$ for
the greatest common divisor and call it $\gcd(f,g)$, by abuse of notation.
  Let $\bar f,\bar g\in\Mla$ be the coprime
pair given uniquely by
$f+g=\gcd(f,g)(\bar f+\bar g)$ and let $\bez_f \oplus \bez_g$ be equal to $\bez_{\bar f} \oplus \bez_{\bar g}$, so that $f \bez_f + g \bez_g = \gcd(f,g)$.
\end{definition} 

\vspace{-10pt} Note that the diagram
$
\def\objectstyle{\scriptstyle}
\def\labelstyle{\scriptstyle}
\vcenter{\xymatrix@-10pt{
  Q \ar[d]_g & \ar[l]_{\tilde f} \ar[d]^{\tilde g} Q\\
  Q & \ar[l]_f Q
}}
$
is cartesian if and only if the diagram 
$
\def\objectstyle{\scriptstyle}
\def\labelstyle{\scriptstyle}
\vcenter{
\xymatrix@-10pt{
  Q \ar[d]_{\bar g} & \ar[l]_{\tilde f} \ar[d]^{\tilde g} Q\\
  Q & \ar[l]_{\bar f} Q
}}
$
is cartesian.  When these diagrams are cartesian, denote the composite $g \tilde f = f \tilde g$ by
$\mathrm{lcm}(f,g)$, giving
the pleasant equation
$|f||g|=|\gcd(f,g)||\,\mathrm{lcm}(f,g)|.$  For each pair $f$ and $g$, we in fact choose a specific pullback of $f$ and $g$, therefore specific maps $\tilde f$ and $\tilde g$; this in turn pins down a particular map $\mathrm{lcm}(f,g)$.

\begin{remark}
We summarize the above discussion and introduce some notation clarifying the relationship of $\{\bar f, \bar g, \tilde f, \tilde g\}$ to $\{f,g\}$.  Given natural numbers $m$ and $n$, there are numbers
\begin{align}
\langle m \rangle := m/\!\gcd(m,n) = \lcm(m,n) / n \nn \\
\langle n \rangle := n/\!\gcd(m,n) = \lcm(m,n) / m \nn
\end{align}
In our noncommutative generalization, there is not one notion here but two, according to whether we generalize the expression ``$m / \gcd(m,n)$" or ``$\lcm(m,n) / n$".  The generalizations are denoted, respectively $\bar{f}$ and $\tilde f$.  We sometimes use the expression $[f / \gcd(f,g)]$ to \emph{mean} $\bar{f}$, and the expression $[\lcm(f,g) / g]$ to \emph{mean} $\tilde f$.  To reiterate then, the endomorphisms $[f / \gcd(f,g)]$ and $[g / \gcd(f,g)]$ are defined to be coprime endomorphisms for which there is a morphisms $\gcd(f,g)$ such that $f = \gcd(f,g) [f / \gcd(f,g)]$ and $g = \gcd(f,g) [g / \gcd(f,g)]$.  Similarly, the endomorphisms $[\lcm(f,g) / g]$ and $[\lcm(f,g) / f]$ are defined to be coprime endomorphisms such that $f [\lcm(f,g) / f] = g [\lcm(f,g) / g]$; that common product is called $\lcm(f,g)$.  Notice that with the above definitions of Bezout endomorphisms, we have the relation $[f / \gcd(f,g)] \bez_f + [g / \gcd(f,g)] \bez_g = 1$.  We sometimes write $\gcd(f,g)$ as $\gcd_{f,g}$ and similarly $\lcm(f,g)$ as $\lcm_{f,g}$.
\end{remark}

\begin{example} \label{ex-gcd}
Propositions~\ref{prop-transfer}.4 and~\ref{prop-vers}.2 describe relations for transfer maps and the Verschiebung.  When the group $G$ in question is the torus $\Tnp$, the order of the double coset $|L_{\tilde b}\backslash L_{a \tilde b}/L_{\tilde a}|$ appearing in those formulae is precisely the volume $|\gcd(a,b)|$.  Here $L_{a}=a^{-1} \zpn / \zpn \subset \Tnp$ is the kernel of the covering associated to the matrix $a$.
\end{example}

Let $K$ be a field containing 
$\Lambda$, and set $b\Lambda=K/\Lambda$.  This quotient $b\Lambda$ should be thought of as the classifying space of $\Lambda$; for example, for the inclusion $\bZ \subset \bR$, the classifying space is $b\bZ = \bR/\bZ$, the circle.  (Totaro has also considered this notion of classifying spaces of rings.) 
When as before $Q$ is a
finitely generated free $\Lambda$-module, let $bQ=b\Lambda\otimes_\Lambda
Q$.  When $f\in\Mla$ there is an induced surjection $bf\colon bQ\to bQ$
with kernel $l_f\cong  Q/f(Q)$ and an isomorphism 
$\phi_f\colon bQ/l_f\to bQ$ induced
by multiplication by $f$, as illustrated by the diagram
$$
\begin{CD}
  0@>>>Q@>>>K\otimes_\Lambda Q@>>> bQ@>>>0\\
  @.@V{f}VV@V{f}V{\cong}V@V{bf}VV\\
  0@>>>Q@>>>K\otimes_\Lambda Q@>>> bQ@>>>0
\end{CD}
$$
(Here the middle map is an isomorphism by Cramer's rule, since $det(f)$ is
invertible in $K$).  As an example, let $\Lambda=\bZ\subseteq\bR=K$, $Q=\bZ$
and let $f$ be multiplication by $m\in\bZ$.  Then 
$bf\colon bQ\to bQ$ is the $m$-fold covering of the circle $bQ=\bR/\bZ$.

\begin{lemma}\label{lem:matvsgp}
  Let $f,g\in\Mla$.  There is an $h\in\Mla$ such that $g=hf$ if and
  only if $l_f\subseteq l_g\subseteq bQ$.  In particular, there is an
  $h\in \Aut_\Lambda(Q)$ such that $g=hf$ if and only if $l_f=l_g$.
\end{lemma}
\begin{proof}
  This follows by a diagram chase involving 
\[\begin{CD}
  0@>>>Q@>>>K\otimes_\Lambda Q@>>> bQ@>>>0\\
  @. @A{g}AA@A{g}A{\cong}A@A{bg}AA\\
  0@>>>Q@>>>K\otimes_\Lambda Q@>>> bQ@>>>0\\
  @.@V{f}VV@V{f}V{\cong}V@V{bf}VV\\
  0@>>>Q@>>>K\otimes_\Lambda Q@>>> bQ@>>>0 
\end{CD} \]
\vspace{-22pt}

\qedhere
\end{proof}

\begin{lemma}\label{lem:diagon}
  If $f\colon\Lambda^k\to\Lambda^n$ is an injection, then there are $\gamma\in
  GL_n(\Lambda)$, $\gamma'\in GL_k(\Lambda)$ such that
  $\gamma'f\gamma$ is represented by a diagonal $k\times k$-matrix
  followed by the standard inclusion $\Lambda^k\subseteq\Lambda^n$.
\end{lemma}
\begin{proof}
  The fundamental theorem for finitely generated modules over a PID
  gives us an isomorphism $\Lambda^n/f(\Lambda^k)\cong
  \Lambda^{n-k}\oplus\bigoplus_{i=1}^k\Lambda/\lambda_i\Lambda$, for some $\lambda_i \in \Lambda$.
  Lifting this isomorphism gives the result.
\end{proof}

\nid If in addition $\Lambda$ is local number ring with maximal ideal generated by
$\pi$, then the diagonal matrix in question can be chosen to have
powers of $\pi$ on the diagonal.

%%%%%

\subsection{Stable splitting of the transfer} \label{sec-stablesplitting}

In the last section we analyzed the interactions of endomorphisms of modules over a PID.  We are of course most interested in the case of the PID $\Lambda = \zp$ and the module $Q = \zpn$.  Associated to a matrix $\alpha \in \mathcal{M}_n = M_n(\zp) \cap GL_n(\bQ_p)$, there is a self-covering of the p-complete n-torus.  The Verschiebung map for $\alpha$ is by definition the Verschiebung map for that induced self-covering.  This Verschiebung was defined in section~\ref{sec-vers} in terms of the transfer map for the projection associated to the covering.

The higher differentials introduced in section~\ref{sec-higherdiff} were built using the stable splitting of the n-torus.  In order to determine how the differentials and the Verschiebung interact, we must therefore describe the relationship of the stable splitting of the n-torus to the transfers associated to projections of the torus---giving such a description is the purpose of this section.

Choose once and for all an equivariant transfer map
$$tr_p\colon (\TT/L_p)_+\to \TT_+.$$
We will need the following fact proved in Hesselholt~\cite{hess-ptypical}.
\begin{lemma}
  With respect to the stable equivalence $\TT_+ \simeq S^0\vee
  S^1$ the composite
$$
  \begin{CD}
    \TT_+@>{\phi^p_+}>>(\TT/L_p)_+@>{tr_p}>>\TT_+
  \end{CD}
$$
is given by the matrix
$$\tau_p=
\left[\begin{matrix}
  p&(p-1)\eta\\0&1
\end{matrix}\right]
$$
\end{lemma}

For $\alpha,\beta\in\Mn$ we analyze the transfer maps 
$tr_{\alpha\beta}^{\beta}\colon 
({\TZ}/L_{\alpha\beta})_+\to ({\TZ}/L_\beta)_+$
as follows.  
Factor
$\alpha=\gamma'\delta\gamma$ where $\gamma',\gamma\in GL_n(\zp)$ and
$\delta$ is a diagonal matrix with diagonal entries powers of $p$.  
The diagram
$$
\begin{CD}
  {\TZ}/L_{\beta}@<{\phi^{\gamma\beta}}<<{\TZ}\\
  @V{\text{proj.}}VV@V{\text{proj.}}VV\\
  {\TZ}/L_{\alpha\beta}@<{\phi^{\gamma\beta}}<<{\TZ}/L_\delta
\end{CD}
$$
is a pullback.  (In fact the horizontal maps are isomorphisms.  Note the absence of $\gamma'$ from the diagram---it was absorbed in an implicit equality in the lower right corner between $\TZ / L_\delta$ and $\TZ / L_{\gamma' \delta}$.  There is a second implicit equality in the upper right corner between $\TZ$ and $\TZ / L_{\gamma^{-1}}$.) The transfers $tr_{\alpha\beta}^{\beta}$ and $tr_{\delta}$ are therefore related by the diagram
$$
\begin{CD}
  ({\TZ}/L_{\beta})_+@<{\phi^{\gamma\beta}_+}<<({\TZ})_+\\
  @A{tr_{\alpha\beta}^{\beta}}AA@A{tr_\delta}AA\\
  ({\TZ}/L_{\alpha\beta})_+@<{\phi^{\gamma\beta}_+}<<({\TZ}/L_\delta)_+
\end{CD}
$$
Moreover, the diagonal transfer $tr_\delta\colon (\Tnp/L_\delta)_+\to (\Tnp)_+$ can be described explicitly as the $p$-completion of smashes of powers of $tr_p$.

Note that if $\eta$ acts trivially, and we view $p_+$ as given by the matrix 
$\left[\begin{smallmatrix}
  1&0\\0&p
\end{smallmatrix}\right]\colon 
\begin{smallmatrix}
  S^0\\S^1
\end{smallmatrix}\! \to \! 
\begin{smallmatrix}
  S^0\\S^1
\end{smallmatrix},
$
the composite $tr_p\phi^p_+$ is given by the adjoint $(p_+)^\dagger=\left[
  \begin{smallmatrix}
    p&0\\0&1
  \end{smallmatrix}\right]
  $.
This continues to be true more generally: if $\alpha\in\Mn$ then $\alpha_+$ can be viewed as a $2^n\times2^n$-matrix through the $n$-fold smash of the stable equivalence $\TT_+\simeq S^0\vee S^1$, and (if $\eta$ acts trivially) the transfer is given in this basis by $(\alpha_+)^\dagger$.  We will in the future occasionally write $(\alpha_+)^\dagger$ for $tr_\alpha\phi^\alpha_+$ even if $\eta$ does not act trivially.

\begin{cor} \label{cor-diagtransfer}
  Let $\delta\in M_n\bZ$ be a diagonal matrix with $p$-power
  entries.  Then, under the stable equivalence $\TT_+ \simeq S^0\vee
  S^1$ the transfer $(\delta_+)^\dagger\colon (\Tnp)_+\to(\Tnp)_+$
is given by the matrix
$$\tau_\delta=\bigotimes_{1\leq j\leq n}
\left[\begin{matrix}
  \delta_{jj}&(\delta_{jj}-1)\eta\\0&1
\end{matrix}\right]
$$
If $\alpha\in\Mn$ then 
\raisebox{1.25pt}{$
\xymatrix{
  S^1\vee\dots\vee S^1\ar[r]^-{\text{inc}} & (\Tnp)_+ \ar[r]^{(\alpha_+)^\dagger} & (\Tnp)_+ \ar[r]^-{\text{proj}} & S^1\vee\dots\vee S^1
}
$}
is given by $\alpha^\dagger$.  If $\eta$ acts trivially then 
\vspace{-8pt}
$$
\begin{CD}
  S^1\vee\dots\vee S^1@>{\text{inc}}>>(\Tnp)_+\\
@V{\text{inc}}VV@V{(\alpha_+)^\dagger}VV\\
(\Tnp)_+@>{(\alpha^\dagger)_+}>>(\Tnp)_+
\end{CD}
$$
commutes in the stable homotopy category.
\end{cor}

%%%%%

\subsection{Relations among the Frobenius, differential, and Verschiebung}

In the one-dimensional case, the Verschiebung is related to the Frobenius and differential by the relations $FV=p$, $VF=V(1)\cdot$, and $FdV=d$.  In Proposition~\ref{prop-vers}, we established multi-dimensional analogs of the first two relations, namely $F^{\alpha} V_{\beta} = |\gcd_{\alpha,\beta}| V_{\tilde \beta} F^{\tilde \alpha}$ and $V_{\alpha} F^{\alpha} = V_{\alpha}(1)\cdot$; (here we have made use of the observation in example~\ref{ex-gcd}).  It remains only to analyze the composite $F^{\alpha} d_{\ell} V_{\beta}$.

\begin{lemma}\label{FdV0}
  Let $\alpha,\beta\in\Mn$ and let $\ell\in M_{n\times k}\zp$. 
Then 
$$F^\alpha d_\ell V_\alpha\colon S^k\smsh \Tn(A)^{L_\beta}\to \Tn(A)^{L_\beta}$$ 
is homotopic to the composite
$$
\xymatrix{S^k\smsh T^{\beta}\ar[r]^-{\sigma\smsh1} &
(\TT_p^k)_+\smsh T^{\beta}\ar[r]^{\ell_+\smsh1}&
(\TT_p^n)_+\smsh T^{\beta}\ar[r]^{(\alpha_+)^\dagger\smsh 1}&
(\TT_p^n)_+\smsh T^{\beta}\ar[r]^-{\phi^\beta\smsh1}&
(\TT_p^n/L_{\beta})_+\smsh T^{\beta}\ar[r]^-{\mu} & T^{\beta}
}
$$
where $(\alpha_+)^\dagger:=tr_\alpha\phi^\alpha_+$, the map $tr_\alpha\colon (\Tnp/L_\alpha)_+\to (\Tnp)_+$ is the transfer,
and as before $T^\beta$ is shorthand for $\Tn(A)^{L_\beta}$.
\end{lemma}
\begin{proof}
  If we show the stable commutativity of 
$$\xymatrix{
{(\Tnp)_+\smsh T^{\beta}}\ar[rr]^{\phi^{\beta}_+\smsh1}&&
{(\Tnp/L_\beta)_+\smsh T^{\beta}}\ar[r]^-{\mu}&T^{\beta}\\
{(\Tnp/L_\alpha)_+\smsh T^{\beta}}\ar[u]^{tr_\alpha\smsh1}&
{(\Tnp)_+\smsh T^{\beta}}\ar[l]_-{\phi^\alpha_+\smsh1}\ar[r]^-{\phi^{\alpha\beta}_+\smsh1}\ar[d]_{1\smsh V_\alpha}&
{(\Tnp/L_{\alpha\beta})_+\smsh T^{\beta}}\ar[u]^{tr_\alpha \smsh1}\ar[d]_{1\smsh V_\alpha}&\\
&{(\Tnp)_+\smsh T^{{\alpha\beta}}}\ar[r]^-{\phi^{\alpha\beta}_+\smsh1}&
{(\Tnp/L_{\alpha\beta})_+\smsh T^{{\alpha\beta}}}\ar[r]^-{\mu}&
T^{{\alpha\beta}}\ar[uu]^{F^\alpha}
}
$$
then we are done: a moment's reflection shows that the two maps in question can be obtained by precomposing
  with 
$\xymatrix{
    S^k\smsh T^{\beta}\ar[r]^-{\sigma\smsh1}&
    {(\Tp[k])_+\smsh T^{\beta}\ar[r]^{\ell_+\smsh1}}
    &(\Tnp)_+\smsh T^{\beta}}.$

The square obviously commutes.  The upper left pentagon commutes because 
$$\xymatrix{{\Tnp}\ar[d]^{\text{proj}}\ar[rr]^{\phi^{\beta}}&&
{\Tnp}/L_\beta\ar[d]^{\text{proj}}\\
{\Tnp}/L_\alpha&{\Tnp}\ar[l]_(.4){\phi^\alpha}\ar[r]^-{\phi^{\alpha\beta}}&
{\Tnp}/L_{\alpha\beta}
}
$$
is a pullback---here of course the horizontal maps are isomorphisms.  The right pentagon commutes by arguments analogous to those in the proof of the one-dimensional $FdV$ relation, specifically by dualizing the Verschiebung and Frobenius and expressing the pentagon as a composite of transfer pullback squares for products of tori, a la diagram 1.5.2 of Hesselholt~\cite{hess-ptypical}.
\end{proof}
\begin{cor}
  In the situation of the lemma, if $\alpha\in GL_n\zp$ then
$$F^\alpha d_\ell V_\alpha=d_{\alpha^{-1}\ell}$$
\end{cor}
\begin{proof} We need only understand $(\alpha_+)^\dagger := tr_\alpha \phi^\alpha_+$.  Here $\phi^\alpha$ is simply $\alpha^{-1}\colon
  \Tnp\to\Tnp/L_\alpha=\Tnp$ and $tr_\alpha$ is the identity.
\end{proof}

\nid This formula may seem surprising, given that $F^\alpha\colon
T^{L_{\alpha\beta}}\to T^{L_\beta}$ and $V_\alpha$ are the identity of
spectra when $\alpha$ is an isomorphism, but follows from the fact
that they are not equivariant, and $d_\ell$ is sensitive to the $\Tnp$-action.

Lemma \ref{FdV0} can more elegantly be rephrased in terms of a less useful variant of the differential: for a stable map $v\colon(\Tp[k])_+\to(\Tnp)_+$ let $D_v$ be the composite
$
\xymatrix{
  (\Tp[k])_+\smsh X \ar[r]^-{v\smsh 1} & (\Tp[n])_+\smsh X \ar[r]^-{\mu} & X,
}
$
so that $d_\ell=D_{(\ell_+)}(\sigma\smsh 1)$.
Then we get that
$$F^\alpha D_v V_\alpha=D_{(\alpha_+)^\dagger v}$$

\begin{cor}\label{cor-FadVa}
  Let $\alpha\in M_n\bZ$. 
Assume multiplication by $\eta$ is nullhomotopic on $T^{L_\beta}$ (for example
if it is an Eilenberg-MacLane spectrum or if $p$ is odd).  If $\ell\in M_{n\times 1}\zp$,  then
$$F^\alpha d_\ell V_\alpha = d_{\alpha^\dagger\ell}$$ and 
$$F^\alpha d_1 V_\alpha 
= \frac{|\alpha|}{\det(\alpha)}\cdot d_1$$  
\end{cor}
\begin{proof}
The first part follows by Corollary~\ref{cor-diagtransfer} and Lemma~\ref{FdV0}.  For the second part, note that, if $\eta=0$, the ``$n$-dimensional part'' of $(\alpha_+)^\dagger$ is a column of zeros except for the first entry which is the unit $|\alpha|/\det(\alpha)$. 
\end{proof}

For differentials $d_\ell$, with $\ell$ an $n \times k$ matrix for $1 < k < n$, we see that $F^\alpha d_\ell V_\alpha$ is given by intermediate matrices of minors.  Note also that the number $\frac{|\alpha|}{\det(\alpha)}\in\zp$ appearing in the above formula is invertible.

\begin{lemma} \label{lem-dFFd}
  Let $\ell\colon\zpk\to\zpn$ and $\alpha\in\Mn$.  Then
  \begin{enumerate}
  \item[1.] $d_\ell F^\alpha=F^\alpha d_{\alpha\ell}$
  \item[2.] $V_\alpha d_\ell=d_{\alpha\ell}V_\alpha$
  \end{enumerate}
\end{lemma}
\begin{proof}The first equation holds because the diagram 
  $$
  \xymatrix{
(\Tnp)_+\smsh T^{{\alpha\beta}}
\ar[r]^{\text{proj}\smsh1}\ar[d]^{\phi^\beta_+\smsh1}&
(\Tnp/L_\alpha)_+\smsh T^{{\alpha\beta}}\ar[d]^{\phi^\beta_+\smsh1}&\\
(\Tnp/L_\beta)_+\smsh T^{{\alpha\beta}}
\ar[r]^{\text{proj}\smsh1}\ar[d]^{1\smsh\text{inc}}&
(\Tnp/L_{\alpha\beta})_+\smsh T^{{\alpha\beta}}
\ar[r]^-{\mu}&T^{{\alpha\beta}}\ar[d]^{\text{inc}}\\
(\Tnp/L_\beta)_+\smsh T^{{\beta}}\ar[rr]^\mu
&&T^{{\beta}}
}
  $$
commutes, and the two sides of the equation can be obtained by precomposing with
$$
\begin{CD}
  S^k\smsh X^{L_{\alpha\beta}}
  @>{\sigma\smsh 1}>>
  \Tp[k]\smsh X^{L_{\alpha\beta}}
  @>{\ell_+\smsh 1}>>
  \Tp[n]\smsh X^{L_{\alpha\beta}}
\end{CD}
$$
The second equation is similar.
\end{proof}

Piecing together our accumulated understanding of the interactions among these operations, we can establish the final relation:

\begin{theorem}
  Let p be an odd prime, and let $\alpha,\beta\in\Mn$, $\ell\colon\zp \to\zpn$.  
Choose a representative $D$ of $\gcd_{\alpha,\beta}$ and coprime 
$\bar\alpha,\bar\beta\in\Mn$ with $\alpha=D \bar\alpha$ 
and $\beta=D \bar\beta$.  
Furthermore, choose a splitting $(s, t)$ to $\bar{\alpha}+\bar\beta$, so that $\bar \alpha s + \bar \beta t = 1$.  
Then
$$F^\alpha d_\ell V_\beta=d_{s D^\dagger\ell}F^{\bar\alpha}V_{\bar\beta}+F^{\bar\alpha}V_{\bar\beta}d_{t D^\dagger\ell}$$
\end{theorem}
\begin{proof}  From Corollary~\ref{cor-FadVa} and Lemma~\ref{lem-dFFd}, we have
  \begin{align*}
    F^\alpha d_\ell V_\beta = \; &  F^{\bar\alpha} F^Dd_\ell V_D V_{\bar\beta}\\
    = \; &  F^{\bar\alpha}d_{D^\dagger\ell}V_{\bar\beta}\\
    = \; &  F^{\bar\alpha}d_{\bar\alpha s D^\dagger\ell+\bar\beta t D^\dagger\ell}V_{\bar\beta}\\
    = \; &  F^{\bar\alpha}(d_{\bar\alpha s D^\dagger\ell}+d_{\bar\beta t D^\dagger\ell})V_{\bar\beta}\\
    = \; &  F^{\bar\alpha}d_{\bar\alpha s D^\dagger\ell}V_{\bar\beta}+ F^{\bar\alpha}d_{\bar\beta t D^\dagger\ell}V_{\bar\beta}\\
    = \; &  d_{s D^\dagger\ell}F^{\bar\alpha}V_{\bar\beta}+F^{\bar\alpha}V_{\bar\beta}d_{t D^\dagger\ell}
  \qedhere
  \end{align*} 
  \end{proof}

\nid Note that by Proposition~\ref{prop-vers}.2, $F^{\bar\alpha} V_{\bar\beta} = V_{\tilde \beta} F^{\tilde \alpha}$, where $\tilde \alpha$ and $\tilde \beta$ are coprime matrices with $\alpha \tilde \beta = \beta \tilde \alpha$.  The right hand side of this $FdV$ relation could thereby be rewritten in terms of $dVF$ and $VFd$.

We summarize the structure of the fixed point spectra $T^{(n)}(A)^{L_{\alpha}}$ in the following omnibus theorem.  We state the results in the homotopy category, though those relations only involving the restriction and Frobenius may be lifted to spectra.

\begin{theorem} \label{thm-relations}
Fix an odd prime p.  Let $A$ be a connective commutative ring spectrum.  For $\alpha \in M_n(\zp) \cap GL_n(\bQ_p)$ an injective endomorphism of $\zpn$, let $L_{\alpha} := \alpha^{-1} \zpn / \zpn \subset \Tnp$ be the corresponding subgroup of the p-adic n-torus.  Denote by $T^{\alpha} := T_{\Tnp}(A)^{L_{\alpha}}$ the fixed points of the n-dimensional topological Hochschild homology of $A$; this is a ring spectrum with multiplication denoted $\mu: T^{\alpha} \sm T^{\alpha} \ra T^{\alpha}$.

There are operators in the stable homotopy category $R_{\alpha}: T^{\beta \alpha} \ra T^{\beta}$ (restriction), $F^{\alpha}: T^{\alpha \beta} \ra T^{\beta}$ (Frobenius), and $V_{\alpha} : T^{\beta} \ra T^{\alpha \beta}$ (Verschiebung).  Moreover, for each p-adic vector $v \in \zpn$, there is an operator $d_v: S^1 \sm T^{\alpha} \ra T^{\alpha}$ (differential).  These $R$, $F$, $V$, and $d$ maps satisfy the following relations.

\begin{enumerate}

\item[1.] 
$R_{\alpha} \mu = \mu R_{\alpha}$; \\
$F^\alpha \mu = \mu F^\alpha$; \\
$d_v (1 \sm \mu) = \mu (d_v \sm 1) + \mu (1 \sm d_v) (\tau \sm 1)$; \\
if $\alpha \in GL_n \zp$ then $F^\alpha = \id$ and $V_\alpha = \id$.  

\item[2.] 
$R_\alpha R_\beta = R_{\alpha \beta}$; \\
$F^{\alpha} F^{\beta} = F^{\beta \alpha}$; \\
$V_\alpha V_\beta = V_{\alpha \beta}$; \\
$d_{v+w} = d_v + d_w$; \\
$d_v (1 \sm d_w) = d_w \tau (d_v \sm 1) \tau$.

\item[3.] 
$R_\alpha F^\beta = F^\beta R_\alpha$; \\
$R_\alpha V_\beta = V_\beta R_\alpha$; \\
$R_\alpha d_v = d_v R_\alpha$.

\item[4.] 
$\mu (V_\alpha \sm 1) = V_\alpha \mu (1 \sm F^\alpha)$; \\ 
$F^\alpha V_\beta = |\gcd_{\alpha,\beta}| V_{[\lcm_{\alpha,\beta} / \alpha]} F^{[\lcm_{\alpha,\beta} / \beta]}$.

\item[5.] 
$d_v F^\alpha = F^\alpha d_{\alpha v}$; \\
$V_\alpha d_v = d_{\alpha v} V_\alpha$; \\ 
$F^\alpha d_v V_\beta = d_{\bez_\alpha \gcd_{\alpha,\beta}^\dagger v} V_{[\lcm_{\alpha,\beta} / \alpha]} F^{[\lcm_{\alpha,\beta} / \beta]} + V_{[\lcm_{\alpha,\beta} / \alpha]} F^{[\lcm_{\alpha,\beta} / \beta]} d_{\bez_\beta \gcd_{\alpha,\beta}^\dagger v}$.

\end{enumerate}

Notation: In the above, we have abbreviated for example $1 \sm V$ as $V$ and $R \sm R$ as $R$ and the like, when no confusion is possible.  In item (1), $\tau$ is a twist, and the equations record the facts that $R$ and $F$ are ring maps, $d_v$ is a derivation, and $F^\alpha$ and $V_\alpha$ only depend on the group $L_\alpha$.  The terms in items (4) and (5) are defined as follows.  Choose matrices $\gcd_{\alpha,\beta}$ and coprime $[\alpha / \gcd_{\alpha,\beta}]$ and $[\beta / \gcd_{\alpha,\beta}]$ such that $\alpha = \gcd_{\alpha, \beta} [\alpha / \gcd_{\alpha,\beta}]$ and $\beta = \gcd_{\alpha, \beta} [\beta / \gcd_{\alpha,\beta}]$.  Next choose ``Bezout" matrices $\bez_\alpha$ and $\bez_\beta$ such that $[\alpha / \gcd_{\alpha,\beta}] \bez_\alpha + [\beta / \gcd_{\alpha,\beta}] \bez_\beta = 1$.  Finally choose coprime matrices $[\lcm_{\alpha,\beta} / \alpha]$ and $[\lcm_{\alpha,\beta} / \beta]$ such that $\alpha [\lcm_{\alpha,\beta} / \alpha] = \beta [\lcm_{\alpha,\beta} / \beta]$; that common product is by definition $\lcm_{\alpha,\beta}$.

\end{theorem}

At the risk of repetition, we follow the lead of Hesselholt and Madsen in taking homotopy groups and repackaging some of this data into a particular kind of pro differential ring.  In the following $p$ is again an odd prime.

\begin{cor} \label{cor-pmdgr}
Associated to a connective commutative ring spectrum $A$, there is a pro multi-differential graded ring 
denoted $TR^{\alpha}_q(A;p)$ and defined as follows.  For each matrix $\alpha \in M_n(\zp) \cap GL_n(\bQ_p)$, we have a group $TR^{\alpha}_q(A;p) := \pi_q(T_{\Tnp}(A)^{L_{\alpha}})$; here $L_\alpha = \alpha^{-1} \zpn / \zpn$.  As $q$ varies these groups form a graded ring.  For each p-adic vector $v \in \zpn$, there is a graded differential $d_v: TR^\alpha_q(A;p) \ra TR^\alpha_{q+1}(A;p)$; these differentials are derivations, are linear  
in the vector $v$, and they graded commute with one another.  The collection $TR^{\alpha}_*(A;p)$ is therefore a multi-differential graded ring.  As $\alpha$ varies these form a pro multi-differential graded ring under the restriction maps $R_\alpha$.

There is a collection of pro-graded-ring operators $F^\alpha: TR^{\alpha \beta}_* \ra TR^\beta_*$, and a collection of pro-graded-module operators $V_\alpha: (F^\alpha)_* TR^{\beta}_* \ra TR^{\alpha \beta}_*$.  Both $F^\alpha$ and $V_\alpha$ depend only on the group $L_\alpha$.  Here $(F^\alpha)_* TR^{\beta}_*$ denotes $TR^{\beta}_*$ with the $TR^{\alpha \beta}_*$-module structure determined by precomposition with $F^\alpha$.  Note that the fact that $V_\alpha$ is a module map is equivalent to Frobenius reciprocity: $V_\alpha(x) \cdot y = V_\alpha (x \cdot F^\alpha(y))$.

These operators are subject to the relations
\begin{enumerate}
\item[1.] $F^\alpha V_\beta = |\gcd_{\alpha,\beta}| V_{[\lcm_{\alpha,\beta} / \alpha]} F^{[\lcm_{\alpha,\beta} / \beta]}$

\item[2.] $d_v F^\alpha = F^\alpha d_{\alpha v}$; $V_\alpha d_v = d_{\alpha v} V_\alpha$

\item[3.] $F^\alpha d_v V_\beta = d_{\bez_\alpha \gcd_{\alpha,\beta}^\dagger v} V_{[\lcm_{\alpha,\beta} / \alpha]} F^{[\lcm_{\alpha,\beta} / \beta]} + V_{[\lcm_{\alpha,\beta} / \alpha]} F^{[\lcm_{\alpha,\beta} / \beta]} d_{\bez_\beta \gcd_{\alpha,\beta}^\dagger v}$
\end{enumerate} 
\end{cor}

It may be that this list of relations, together with a few discussed earlier concerning higher differentials, is in a useful sense complete.  In this case, it is clear how to utilize and abstract the structure seen in the corollary: add a map (compatible with the operators) from the Burnside-Witt vectors~\cite{dresssiebeneicher, graham-witt} into the degree zero part of the pro multi-differential graded ring---we would call the resulting structure a Burnside-Witt complex.  (Compare Hesselholt-Madsen~\cite{hm-mixedchar}.)  It would be worth studying the initial such complex, which serves the role of a higher analog of the de Rham-Witt complex.  In future work we will make precise and investigate these notions of Burnside-Witt complex and ``de Rham-Burnside-Witt" complex.

%%%%%%%%%%%

\section{Adams operations on covering homology} \label{sec-adamsops}

Recall that in the end we are interested in covering homologies $TC^J(A)$, that is in versions of higher topological cyclic homology---these are homotopy limits of the fixed points $T^{(n)}(A)^{L_\alpha}$ under collections of restriction and Frobenius operators associated to submonoids $J$ of the self-isogenies of the p-adic n-torus.  There are groups of Adams operations acting on the fixed points $T^{(n)}(A)^{L_\alpha}$, and for appropriate choices of submonoids $J$, these operations survive the homotopy limit to produce interesting actions on the covering homology $TC^J(A)$.  Indeed, these actions will play a crucial role in our future investigation of chromatic phenomena arising from higher topological cyclic homology.

\subsection{Identification of the Frobenius and restriction categories} \label{sec-identfrobres}

We begin by explicitly identifying the structure of the indexing categories for Frobenius and restriction operators, in the case of higher topological Hochschild homology based on the p-adic n-torus.  Recall from section~\ref{sec-covhom} that $\Mn := M_n(\zp) \cap GL_n(\bQ_p)$ and $Frob_{\Mn}$ and $Res_{\Mn}$ are the subcategories of the twisted arrow category $\mathcal Ar_{\Mn}$ consisting respectively of the maps $\alpha^*$ and $\alpha_*$.

\begin{lemma}\label{lem:RFcats}
Two objects in $Frob_{\mathcal{M}_n}$ are connected by at most
one morphism; similarly for $Res_{\mathcal{M}_n}$.  The group $GL_n(\zp)$ acts both on the left and
the right of both $Frob_{\mathcal{M}_n}$ and $Res_{\mathcal{M}_n}$.
\end{lemma}
\begin{proof}
  Note that since the matrices in $\mathcal{M}_n$ are invertible over $\bQ_p$ we
  have cancellation:  $fgh=fg'h\Rightarrow g=g'$.  Hence, in both
  $Res_{\mathcal{M}_n}$ and $Frob_{\mathcal{M}_n}$ there is at most one arrow between two
  objects. If $a^*\colon ab\to b\in Frob_{\mathcal{M}_n}$ and $g\in
  GL_n(\zp)$ we set $g\cdot a^*=(gag^{-1})^*\colon gab\to gb$ and
  $a^*\cdot g=a^*\colon abg\to bg$.  The other action is similar.
\end{proof}

\nid Notice that whereas there is at most one map between two objects in
$Res_{\mathcal{M}_n}$ or $Frob_{\mathcal{M}_n}$, in the ``amalgamation'' $\mathcal Ar_{\mathcal{M}_n}$ there can be many.  

\begin{lemma}\label{lem:RFcats2}
For $\alpha \in \Mn$, the assignment $\alpha \mapsto L_\alpha=\alpha^{-1}\zpn/\zpn$ defines
an equivalence of categories 
$$
  Frob_{\Mn}\we GL_n(\zp)\backslash Frob_{\Mn}\cong Sub_{\Tnp}^{\op}
$$
from $Frob_{\Mn}$ to the opposite of the category $Sub_{\Tnp}$ of finite subgroups of
$\Tnp = \bR^n_p/\zpn$ and inclusions, and induces an isomorphism between $GL_n(\zp)\backslash Frob_{\Mn}$ and $Sub_{\Tnp}^{\op}$.

Dually, the assignment $\beta\mapsto \beta\zpn$ defines an
equivalence
$$Res_{\Mn}\we Res_{\Mn}/GL_n(\zp)\cong Op_{\zpn}
$$ between $Res_{\Mn}$ and the category $Op_{\zpn}$ of open subgroups of $\zpn$ and inclusions, and induces an isomorphism between the orbit category $Res_{\Mn}/GL_n(\zp)$ and $Op_{\zpn}$.
\end{lemma}
\begin{proof}
Since all finite subgroups of $\Tnp$ are of the form
$L_\alpha$ for some $\alpha$, Lemma~\ref{lem:matvsgp} is just another way of stating that $\alpha\mapsto L_\alpha$ induces an isomorphism between $GL_n(\zp)\backslash Frob_{\Mn}$ and $Sub_{\Tnp}^{\op}$.

Choose a section $\sigma$ of the projection $p\colon \Mn\to GL_n(\zp)\backslash\Mn$.  For every $\beta\in\Mn$ there is a unique $\gamma_\beta\in GL_n(\zp)$ such that $\sigma(p(\beta))=\gamma_\beta\beta$.  Let $\Sigma\colon GL_n(\bZ_p)\backslash Frob_{\Mn}\to Frob_{\Mn}$ be defined by $\Sigma (p(\beta))=\sigma(p(\beta))=\gamma_\beta\beta$ and $\Sigma(p(\beta)\to p(\alpha\beta))=(\gamma_\beta\beta\to \gamma_{\alpha\beta}\alpha\beta=(\gamma_{\alpha\beta}\alpha\gamma_{\beta}^{-1})\gamma_\beta\beta)$, giving a section to the projection to the orbit category, with $\gamma$ as the natural isomorphism establishing the equivalence.

The last statement is dual.
The analog of Lemma~\ref{lem:matvsgp} for $Res_{\Mn}$ is the statement that
two open subgroups $\alpha\zpn$ and $\beta\zpn$ of $\zpn$  are equal if and
only if $\beta^{-1}\alpha\in GL_n(\zp)$, and more generally
$\alpha\zpn\subseteq\beta\zpn$ if and only if
$\beta^{-1}\alpha\in\Mn$.
\end{proof}

Henceforth we abbreviate $Sub_{\Tnp}^{\op}$ by $Sub$.  In our later computation of topological Frobenius homology, we will need the following result concerning the orbit structure of this category of subgroups.

\begin{cor}\label{tranact}
  Let $K\in Sub$.  The orbit of $K$ under the $GL_n(\zp)$-action on
  $Sub$ consists of the subgroups that are abstractly group-isomorphic
  to $K$.  

The stabilizer of $K$  has finite index in $GL_n(\zp)$, and if $K\subseteq L_{p^l \id_n}$ for some $l$, then the entire orbit will
  consist of subgroups of $L_{p^l \id_n}$.
\end{cor}
\begin{proof}
  Clearly, all elements in the orbit of $K$ are abstractly isomorphic
  to $K$.  Conversely, assume $K'\in Sub$ and given a group
  isomorphism $x\colon K\cong K'$.  Choose matrices
  $\alpha,\alpha'\in\Mn$ such that $K=L_\alpha$ and $K'=L_{\alpha'}$
  and consider the diagram
$$
\xymatrix{
0\ar[r]&\zpn\ar@{.>}[d]^{\gamma}\ar@{^{(}->}[r]&
\alpha^{-1}\zpn\ar@{.>}[d]\ar@{->>}[r]& 
K\ar[r]\ar[d]^x_{\cong}&0\\
0\ar[r]&\zpn\ar@{^{(}->}[r]&
(\alpha')^{-1}\zpn\ar@{->>}[r]& 
K'\ar[r]&0
}
$$
By Nakayama's lemma the dotted middle vertical map exists and is an
isomorphism, and so the dotted left vertical map exists, and is
represented by a matrix, say $\gamma$.  This means that
$K'=L_{\alpha\gamma^{-1}}=\gamma\cdot K$.

Lemma \ref{lem:matvsgp} implies that $L_{\alpha}\subseteq L_{p^l \id_n}$
if and only if $p^l\alpha^{-1}\in\Mn$, and so certainly such an $l$
exists.  Furthermore, since $p^l\gamma\alpha^{-1}=\gamma\cdot
p^l\alpha^{-1}$ we also have that $L_{\alpha\gamma^{-1}}\subseteq
L_{p^l \id_n}$.
The stabilizer has finite
  index because there are only finitely many subgroups of $L_{p^l \id_n}$.
\end{proof}

\begin{remark}
Occasionally we will be exclusively interested in homotopy limits over subcategories of $Frob_{\Mn}$, and for those purposes we can avoid the $p$-adic integers, as follows.  Let $Frob_{\Mn \bZ}$ be the full subcategory of $Frob_{\Mn}$ with objects \emph{integral} matrices with determinants a power of $p$.  The inclusion $Frob_{\Mn \bZ}\subseteq Frob_{\Mn}$ is an equivalence of categories, and the assignment $\alpha\mapsto L_\alpha$ induces an isomorphism between $GL_n(\bZ)\backslash Frob_{\Mn \bZ}$ and the category $Sub_{\Cp}^{\op}$ of finite subgroups of $\Cp$.
\end{remark}

\subsection{Actions on inverse systems}

In this section, we describe actions by groups $I \subseteq GL_n(\zp)$ on the covering homology $TC^J(A)$, for submonoids $J \subseteq \Mn$ of isogenies of the torus.

\begin{prop}
Let $I\subseteq GL_n(\zp)$ be a subgroup and $J\subseteq\Mn$ a submonoid such
that if $\gamma\in I$ and $\beta\in J$ then $\gamma\beta\gamma^{-1}\in
J$.  In this case the group $I$ acts on $TC^J(A) := \holim_{\beta\in\mathcal Ar_J}\Tn(A)^{L_\beta}$ by sending $\gamma \in I$ to the operator $R_\gamma F^{\gamma^{-1}}$.
\end{prop}
\begin{proof}
Observe that if $\gamma\in I$ and $\alpha,\beta\in J$ then
$$
\begin{CD}
  T^{{\gamma\alpha\beta\gamma^{-1}}}@>{F^{\gamma^{-1}}}>{=}>
T^{{\gamma\alpha\beta}}@>{R_\gamma}>>T^{{\alpha\beta}}\\
   @V{F^{\gamma\alpha\gamma^{-1}}}VV@V{F^\alpha}VV@V{F^\alpha}VV\\
    T^{{\gamma\beta\gamma{-1}}}@>{F^{\gamma^{-1}}}>{=}>
T^{{\gamma\beta}}@>{R_\gamma}>>T^{{\beta}}
\end{CD}
$$
and
$$
\begin{CD}
  T^{{\gamma\alpha\beta\gamma^{-1}}}@>{F^{\gamma^{-1}}}>{=}>
T^{{\gamma\alpha\beta}}@>{R_\gamma}>>T^{{\alpha\beta}}\\
    @V{R_{\gamma\beta\gamma^{-1}}}VV@V{R_{\gamma\beta\gamma^{-1}}}VV@V{R_\beta}VV\\
    T^{{\gamma\alpha\gamma^{-1}}}@>{F^{\gamma^{-1}}}>{=}>
T^{{\gamma\alpha}}@>{R_\gamma}>>T^{{\alpha}}
\end{CD}
$$
commute.
\end{proof}

The question is how to choose appropriate $I$ and $J$.  The intersection between $I$ and $J$ should be kept as small as possible, since any element in the intersection will already be in the indexing category of the homotopy limit, and so the action will be trivial.  The minimal $J$ of interest is the free submonoid generated by $p$, which consists of the matrices $p^k \id_n$ for $k\in\bN$.  Any $\alpha\in\Mn$ commutes with these matrices, so in this case we may let $I=GL_n(\zp)$:

\begin{cor} \label{cor-diagcyc}
  Let $\Delta\subseteq \Mn$ be the free submonoid generated by $p \cdot \id_n$.  The group $GL_n(\zp)$ acts on the ``topological diagonal cyclic homology" 
$$TC^{\Delta}(A) = \holim_{p^k \id_n \in\mathcal Ar_\Delta}\Tn(A)^{L_{p^k \id_n}}.$$
\end{cor}

\begin{remark}
When $\beta \in GL_n \zp$, the restriction map $R_\beta$ is an isomorphism of $\Tnp$-spectra, which can be rather confusing.  In keeping track of the action, the reader may find consolation in the fact that the diagram
$$\xymatrix{
{\TT_p^{n}\times\TT_p^{n}/L_{\alpha\beta}}\ar[d]^{1\times\phi_\beta}&
{\TT_p^{n}/L_{\alpha\beta}\times\TT_p^{n}/L_{\alpha\beta}}
\ar[l]_-{\phi_{\alpha\beta}\times 1}\ar[r]^-{\mu}\ar@{.>}[d]^{\phi_\beta\times\phi_\beta}&
{\TT_p^{n}/L_{\alpha\beta}}\ar[d]^{\phi_\beta}\\
{\TT_p^{n}\times\TT_p^{n}/L_\alpha}&
{\TT_p^{n}/L_\alpha\times\TT_p^{n}/L_\alpha}
\ar[l]_-{\phi_{\alpha}\times 1}\ar[r]^-{\mu}&
{\TT_p^{n}/L_\alpha}}
$$
commutes, where the solid vertical maps keep track of the reindexing implemented by the restriction map, while the horizontal maps keep track of the actions.  At root, the restriction map for an invertible $\beta$ is equal to the functorial action by $\beta: \Tnp \ra \Tnp$, namely the map $T_{\beta}: T_{\Tnp}(A) \ra T_{\Tnp}(A)$, and it is this action that survives to $TC^{\Delta}(A)$.  Note that the target $T_{\Tnp}(A)$ of the map $T_{\beta}$ has its $\Tnp$-action twisted by $\phi_\beta$, so that $T_{\beta}$ remains equivariant.
\end{remark}

\subsection{Examples of group actions on covering homology} \label{sec-exampleactions}

An interesting class of examples arises as follows.  Let $B$ be a
$\zp$-algebra that is free and finitely generated of rank $n$ as a $\zp$-module.  Choose a basis for $B$ and let $M_B=B\cap (B\otimes\bQ)^*$, that is the elements of $B$ that are invertible over $\bQ$.  Using the basis, the monoid $M_B$ is naturally identified as a submonoid of $\Mn$.  We may then consider
$$TC^B(A) := TC^{M_B}(A) = \holim_{\beta\in\mathcal Ar_{M_B}} T^{(n)}(A)^{L_\beta}$$

\begin{example}
Consider the $2$-adic Gaussian integers $B = \bZ_2[i]$.  Up to multiplication by units of $B$,
any element in $M_B$ is a power of $1+i \in M_B$.  Note that $\pi_0 TR^B(A) = W_B(\pi_0 A)$, where $W_B(\pi_0 A)$ is the Burnisde-Witt vectors for $\pi_0 A$ over the module $B \cong \bZ_2^{\times 2}$~\cite{bcd}.  In this example of the $2$-adic Gaussian integers, $TC^B(A)$ is the homotopy limit of $TR^B(A) \overset{F^{1+i}}{\underset{\id}{\rightrightarrows}} TR^B(A)$.  As a result, we have $\pi_{-1} TC^B(A) = W_B(\pi_0 A)_{F^{1+i}}$; this last expression denotes the Frobenius coinvariants of the Burnside-Witt ring.  The reader should contrast this example with the diagonal cyclic homology $TC^{\Delta}(A)$ of Corollary~\ref{cor-diagcyc}.  Both examples come from fixed points of the 2-adic 2-toral topological Hochschild homology, and in both examples the indexing category is ``linear" in the sense that it is generated by a single pair of Frobenius and Restriction operators.  However, in the diagonal cyclic case, $\pi_{-1} TC^{\Delta}(A) = W_{\bZ_2^{\times 2}}(\pi_0 A)_{F^{2 \cdot \id_2}}$.  Note that in the Gaussian integers example, $F^2 = F^{2i} = (F^{1+i})^2$---the two coequalizers, with respect to $F^{1+i}$ and $F^{2 \cdot \id_2}$, are not in general the same.  The reader is invited to make these calculations explicit in the case of $A=\bZ$, using the equivalence $W_{\bZ_2^{\times 2}}(\bZ) \cong K_0(\text{almost finite } \bZ_2^{\times 2}\text{-sets})$.  Other imaginary quadratic orders provide similarly interesting examples.
\end{example}

We get an action of the group $\Aut_{\zp\text{-algebra}}(B)$
of $\zp$-algebra automorphisms of $B$ as follows.  Using the chosen basis of $B$, view $\Aut_{\zp\text{-algebra}}(B)$ as a submonoid of $GL_n(\zp)$; an automorphism $g$ corresponds to a matrix denoted $x_g$.  As above, identify $M_B$ with the corresponding submonoid of $\Mn$.  The automorphism $g \in \Aut_{\zp\text{-algebra}}(B)$ acts on the submonoid $M_B$ by $g(\beta) = x_g \beta x_g^{-1}$, where $\beta \in M_B \subset \Mn$.  The diagram
$$
  \begin{CD}
    \Tn(A)^{L_{g(\alpha\beta\gamma)}}@>{F^{x_g}R_{x_g^{-1}}}>>\Tn(A)^{L_{\alpha\beta\gamma}}\\
	@V{F^{g\gamma}R_{g\alpha}}VV@V{F^{\gamma}R_{\alpha}}VV\\
    \Tn(A)^{L_{g(\beta)}}@>{F^{x_g} R_{x_g^{-1}}}>>\Tn(A)^{L_{\beta}}
  \end{CD}
$$
commutes (because the diagram
$$
\begin{CD}
  B@>{b\mapsto gb}>>B\\@V{\beta\cdot}VV@V{g(\beta)\cdot}VV\\B@>{b\mapsto gb}>>B
\end{CD}
$$commutes).  Hence we have a natural transformation 
$$
\xymatrix@R=0pt@C-5pt{
 {\mathcal Ar_{M_B}}\ar[rrdd]^-{\; \; \; \; \beta\mapsto\Tn(A)^{L_\beta}}\ar[dddd]_{{g}}
 && \\
 &\ar@{=>}[dddl]_(.35){RF^g}& \\
 && \SpecCat\\
 &&\\
 {\mathcal Ar_{M_B}}\ar[uurr]_{\; \; \; \; \beta\mapsto\Tn(A)^{L_\beta}} &&
}
$$
where $RF^g:=R_{x_g}F^{x_g^{-1}}$.
The automorphism group $\Aut_{\zp\text{-algebra}}(B)$ therefore acts on $TC^B(A)$.

\begin{example} . 
  \begin{enumerate}
  \item[1.] Let $G$ be a group of order $n$ and $B=\zp[G]$.  In this case $G$ acts
    on $TC^B(A)$.
  \item[2.] Let $\mathcal O_K$ be the ring of integers in the degree $n$ extension $K$ of $\bQ_p$. The Galois group $G=Gal(K/\bQ_p)$ acts on $TC^{\mathcal
  O_K}(A)$.  For instance, if $n=2$ and $p\neq 1\! \mod 4$, then we have ``complex
  conjugation'' on $TC^{\mathcal O_K}(A)$.
\item[3.] Let $\mathcal O$ be a maximal order in a division algebra of rank $n^2$ over
  $\bQ_p$.  The algebra automorphisms of $\mathcal
  O$---and in particular the units $\mathcal O^*$---act on $TC^{\mathcal
  O}(A)$.  The reader who may wonder at the curious appearance of the $n^2$-fold iterated topological Hochschild homology in this example is invited to compare with the origin of $n^2$-fold abelian varieties in the chromatic level $n$ version of Behrens and Lawson's topological automorphic forms~\cite{behrenslawson-taf}.
  We will give a detailed study of this example in particular in future work.
  \end{enumerate}
\end{example}

%%%%%%%%%%%

\section{Calculation of $TR^{(n)}$ for the sphere}

Recall that $TR^{(n)}(A)$ is the homotopy limit over the restriction operators acting on fixed points of higher topological Hochschild homology $T^{(n)}(A)$.  This ``topological restriction homology" is an important way-station en route to computations of topological cyclic homology.  In this section we calculate topological restriction homology for the sphere spectrum:

\begin{prop}
  There is an equivalence
$$TR^{(n)}(\ess)\simeq\prod_{\mathcal O\subseteq \zpn} B(\zpn/\mathcal O)_+$$
where the product varies over the open subgroups $\mathcal
O\subseteq\zpn$.
\end{prop}

Our computation of $TR^{(n)}(\ess)$, as well as our computation of $TF^{(n)}(\ess)$ in section~\ref{sec-tfnsphere}, will be based on a comparison of the fixed points of topological Hochschild homology with the $K$-theory of equivariant finite sets.  Below we give a detailed account of the homotopy type and the functorial properties of these $K$-theory spectra.  Then we relate the $K$-theory of equivariant sets to the fixed points of the equivariant sphere spectrum.  Finally we note that the topological Hochschild homology of the sphere is the equivariant sphere spectrum, and thereby complete the description of $TR^{(n)}(\ess)$.

\subsection{The K-theory of finite $G$-sets} \label{sec-ktfinite}

For $G$ a finite group, let ${Sets}^G$ denote the category whose objects are finite $G$-sets and whose morphisms are $G$-equivariant isomorphisms.  We begin by reviewing the structure of the $K$-theory spectra $K({Sets}^G)$ as described in Segal~\cite{segal-niceproceedings} or Carlsson~\cite{carlsson-topology}.  

Note that any transitive $G$-set $X$ is isomorphic to one of the form $G/K$, where $K$ is a subgroup of $G$.  The conjugacy class of $K$ is determined by $X$, and conversely determines $X$ up to isomorphism.  Any finite $G$-set $X$ has a decomposition 
$X \cong \coprod _{i = 1} ^sG/K_i
$,
and $s$ and the collection of conjugacy classes of subgroups (with multiplicities) are determined uniquely by $X$ and in turn determine $X$ up to isomorphism of $G$-sets.  It follows that there is an equivalence of categories 
$$  {Sets}^G \cong \prod _{[K]} S[K]
$$
where the product is over the conjugacy classes of subgroups of $G$, and where $S[K]$ denotes the full subcategory of ${Sets}^G$ on $G$-sets $X$ for which the stabilizer of every $x \in X$ is in the conjugacy class $[K]$.  Each $S[K]$ is a symmetric monoidal subcategory of ${Sets}^G$.  The automorphism group of the object 
$(G/K)^n$ is the wreath product $\Sigma _n \wr W_G(K)$, where $W_G(K) $ denotes the quotient group $N_G(K)/K$, with $N_G (K)$ the normalizer of $K$ in $G$.  It is now an easy consequence of the Barratt-Priddy-Quillen Theorem~\cite{priddy-symmetric} that the spectrum associated to the symmetric monoidal category $S[K]$ is the suspension spectrum $\Sigma ^{\infty} BW_G(K)_+$.  There is a corresponding decomposition of spectra 
$$
K({Sets}^G) \cong \bigvee _{[K]} BW_G(K)_+ 
$$
where again the wedge sum is over the conjugacy classes of subgroups of $G$.  We now restrict our attention to finite abelian $G$, in which case the $K$-theory is simply
$$K({Sets}^G)\we \bigvee_{K\subseteq G} B(G/K)_+
$$

We will need to understand the functorial behavior of this decomposition under group homomorphisms.  Let $i: G^{\prime} \hookrightarrow G$ be the inclusion of a subgroup.  A quotient $G/K$, regarded as a $G^{\prime}$-set by restriction of the action along $i$, decomposes as 
$G/K \cong \coprod _{G/K \cdot G^{\prime}} G^{\prime}/ K \cap G^{\prime} 
$.
This decomposition is reflected in $K$-theory as follows.
The induced map $i^* : K({Sets}^G) \rightarrow K({Sets}^{G'})$ sends the $K$-theory factor $K(S[K])$ corresponding to the subgroup $K$ to the factor $K(S[K \cap G^{\prime}])$ corresponding to the subgroup $K \cap G^{\prime}$.  Moreover, the map on that factor is the transfer $B(G/K)_+ \rightarrow B(G^{\prime}/K \cap G^{\prime})_+$
associated to the inclusion $G^{\prime}/K \cap G^{\prime} \hookrightarrow G/K$.

The $K$-theory of finite sets is functorial not only for inclusions of groups, but also for surjective homomorphisms of groups.  If $f\colon G\fib G'$ is a surjection with kernel $H$ there is a ``restriction map'' $f_!\colon {Sets}^G\to {Sets}^{G'}$ sending the finite $G$-set $X$ to its fixed points $X^H$.  In particular $f_!$ sends the $G$-set $G/K$ to the empty set if $H\not\subseteq K$ and to $G'/f(K)$ if $H\subseteq K$.  Under the equivalence above, this map $f_!$ corresponds to the projection
$\bigvee_{K\subseteq G} B(G/K)_+\to\bigvee_{J\subseteq G'} B(G'/J)_+
$
induced by the isomorphism $G/K\cong G'/f(K)$ if $H\subseteq K$, and
the trivial map if not.

For those as perplexed as the authors by the classical conflict of
terminology, we offer the following dictionary between the language of
the Burnside ring and that of the Witt structure on topological Hochschild homology. Here
$i\colon H\subseteq G$ is the inclusion of a subgroup, and $f: G \ra G/H$ is the projection.

\noindent\begin{center}
  \begin{tabular}
{|c|c|}
\hline
  {\bf $G$-sets}&{\bf Topological Hochschild  Homology} \\
\hline
Restriction to action by subgroup &Frobenius = inclusion of fixed
points\\
$i^*\colon {Sets}^G\to {Sets}^H$&$F\colon\Tn(A)^G \ra \Tn(A)^H$\\
\hline
Inducing up action $X\mapsto G\times_HX$&{\Ver} = transfer\\
$i_*\colon {Sets}^H\to {Sets}^G$& $V\colon\Tn(A)^H \ra \Tn(A)^G$\\
\hline
Taking fixed points $X\mapsto X^H$& Restriction\\
$f_!\colon {Sets}^G\to {Sets}^{G/H}$ &$R\colon T_{\Tnp}(A)^G \to T_{\Tnp / H}(A)^{G/H}$\\
\hline
\end{tabular}
\end{center}

\subsection{K-theory and the equivariant sphere spectrum} \label{sec-kthyequivsphere}

When $G$ is a finite group, there is an equivalence between the $K$-theory of the category of finite $G$-sets and the $G$-fixed points of the equivariant sphere spectrum---this is the equivariant Barratt-Priddy-Quillen theorem.  The topological restriction homology $TR^{(n)}(\ess)$ is a homotopy limit of $G$-fixed points of the sphere for varying $G$.  To express $TR^{(n)}(\ess)$ in terms of $K$-theory, we need an equivalence between the fixed points of the sphere and the $K$-theory of finite sets that intertwines the restriction operation on the sphere with a fixed point operation on sets.  In this section we construct an explicit chain of equivalences that has the necessary functoriality properties:
$$\ess^G \xra{\simeq} NT_G \xla{\simeq} NB_G \xra{\simeq} NC_G \xla{\simeq} ND_G \xra{\simeq} K({Sets}^G)$$
The intermediate spectra $NT_G$, $NB_G$, $NC_G$, and $ND_G$ are described below.

We use $\Gamma$-spaces as our model of spectra, and generally follow the configuration space viewpoint Segal developed in his proof of the non-equivariant Barratt-Priddy-Quillen theorem~\cite{segal-config}.  Let $\ess^G$ denote the $\Gamma$-space taking a finite based set $F$ to the space $\hocolim_V \Map_G(S^V, F \sm S^V)$---this is the $G$-fixed point spectrum of the sphere.  As before ${Sets}^G$ is the symmetric monoidal category of finite $G$-sets, and let $\bar{K}({Sets}^G)$ denote the associated $\Gamma$-category; that is, the category $\bar{K}({Sets}^G)(F)$ has objects the equivariantly $(F\backslash *)$-labeled finite $G$-sets, with morphisms the equivariant label-preserving maps.  Recall that the $K$-theory spectrum $K({Sets}^G)$ is the $\Gamma$-space defined by the levelwise nerve of $\bar{K}({Sets}^G)$.

The $\Gamma$-category $\bar{K}({Sets}^G)$ encodes abstract finite $G$-sets.  The topological $\Gamma$-category $C_G$ will encode configurations of finite $G$-sets in $G$-representations, and the topological $\Gamma$-category $T_G$ will encode the self-maps of $G$-spheres determined by the Thom construction on the normal bundles of those configurations of $G$-sets.  The topological $\Gamma$-categories $B_G$ and $D_G$ will provide technical bridges between the Thom category, the configuration category, and the abstract $G$-set category.  

We proceed directly to the definitions.  Fix a complete universe $\cU$ of orthogonal $G$-representations.  The objects of the category $D_G(F)$ consist of triples $(X, V, \phi)$, where $X$ is an equivariantly $(F \backslash *)$-labeled finite $G$-set, $V$ is a subrepresentation of $\cU$, and $\phi: X \ra V$ is an equivariant embedding.  A morphism of $D_G(F)$ from $(X, V, \phi)$ to $(Y, W, \psi)$ is an inclusion $V \subset W$ of subrepresentations of $\cU$ and a label-preserving equivariant isomorphism $\gamma: X \ra Y$ such that $\psi \gamma = \phi$.  The map $D_G \ra \bar{K}({Sets}^G)$ of topological $\Gamma$-categories takes $D_G(F)$ to $\bar{K}({Sets}^G)(F)$ by sending the triple $(X, V, \phi)$ to the labeled $G$-set $X$.  The resulting map $ND_G \ra K({Sets}^G)$ of $\Gamma$-spaces is a levelwise equivalence: the homotopy fiber of $D_G(F) \ra \bar{K}({Sets}^G)(F)$ is a contractible equivariant embedding category, and the equivalence follows by topological Quillen Theorem A.

Next we define the configuration category $C_G$.  The objects of $C_G(F)$ are pairs $(C,V)$, where $V$ is a subrepresentation of $\cU$ and $C$ is an equivariantly $(F \backslash *)$-labeled finite $G$-subset of $V$.  A morphism of $C_G(F)$ from $(C,V)$ to $(D,W)$ is an inclusion $V \subset W$ of subrepresentations taking $C$ label-preservingly isomorphically onto $D$.  The equivalence $C_G(F) \la D_G(F)$ takes the triple $(X, V, \phi)$ to the pair $(\phi(X), V)$.  The category $B_G$ is an enlargement of $C_G$ that includes balls around the configurations.  The objects of $B_G(F)$ are pairs $(A,V)$, where again $V$ is a subrepresentation of $\cU$, and $A$ is an equivariant collection of disjoint open metric balls in $V$ whose components are equivariantly $(F \backslash *)$-labeled.  A morphism from $(A,V)$ to $(B,W)$ is an inclusion $V \subset W$ of subrepresentations taking $A$ into $B$ by a proper map and inducing a label-preserving isomorphism of the set of centers of the $A$-balls to the set of centers of the $B$-balls.  The equivalence $B_G(F) \ra C_G(F)$ takes the pair $(A,V)$ to the pair $(A^{\text{cent}},V)$, where $A^{\text{cent}}$ denotes the set of centers of the balls of $A$.

The topological $\Gamma$-category $T_G$ is defined as follows.  The objects of $T_G(F)$ consist of pairs $(V,f)$, where $V$ is a subrepresentation of $\cU$, and $f: S^V \ra F \sm S^V$ is an equivariant map.  A morphism from $(V,f)$ to $(W,g)$ is an inclusion $V \subset W$ of subrepresentations together with a distinguished equivariant homotopy class of equivariant homotopies relative to $f$ from $g$ to $\Sigma^{W-V} f$.  The map $T_G(F) \la B_G(F)$ takes a pair $(A,V)$, where $A$ is a labeled collection of balls in $V$, to the pair $(V,f)$, where $f$ is the Pontryagin-Thom construction on the open inclusion $A \subset V$, modified so that each ball of $A$ maps by a standard linear isomorphism to $V$ and so that a ball of $A$ labeled by $p \in F$ maps to the $p$-component of the target $F \sm S^V$; a morphism of $B_G(F)$ from $(A,V)$ to $(B,W)$ maps to the morphism of $T_G(F)$ given by the inclusion $V \subset W$ together with the homotopy class of a standard linear homotopy from the Pontryagin-Thom construction on $B \subset W$ to the $(W-V)$-suspension of the Pontryagin-Thom construction on $A \subset V$.  Unlike the equivalences between $NB_G$, $NC_G$, $ND_G$ and $K({Sets}^G)$, the map $NT_G \la NB_G$ is not a levelwise equivalence, but it is a stable equivalence of $\Gamma$-spaces.

The last map $\ess^G \ra NT_G$ does not come from a map of $\Gamma$-categories, but is a levelwise equivalence.  Chose the standard model for the homotopy colimit of a functor $X$ from a diagram $I$ to spaces, namely
$$\hocolim X = \colim \Big( \coprod_{i \ra j} X(i) \otimes N(j/I)^{\op} \rra \coprod_i X(i) \otimes N(i/I)^{\op} \Big)$$
The functor in question is $X_{\ess^G(F)}(V)=\Map_G(S^V, F \sm S^V)$, defined on the diagram of subrepresentations of $\cU$.  The map $\ess^G(F) \ra NT_G(F)$ sends the simplex $\{(V,f) \otimes (V \subset V_1 \subset \cdots \subset V_k)\}$ of $\Map_G(S^V, F \sm S^V) \otimes N(V/I)^{\op}$ to the simplex $\{(V, f) \ra (V_1, \Sigma^{V_1 - V} f) \ra \cdots \ra (V_k, \Sigma^{V_k - V} f)\}$ of $NT_G(F)$---in the latter simplex the morphisms are given by the homotopy classes of the identities.

All the above constructions are functorial with respect to changing the group $G$.  That is, for $i: H \ra G$ an inclusion of finite groups, and for $f: G \ra G'$ a surjection of finite groups, each of the equivalences commutes with the maps $i^*$, $i_*$, and $f_!$.

\subsection{The homotopy limit over the restrictions}

We now have all the ingredients to express the restriction system of fixed points of topological Hochschild homology in terms of a system of $K$-theory functors of categories of $G$-sets, and thereby to compute topological restriction homology.

\begin{prop} \label{prop-tgkg}
  Let $G$ be a finite group and $S$ a free $G$-space.  There is a weak equivalence $\dT[S](\ess)^G\simeq \ess^G$.  When $f\colon G \twoheadrightarrow G'$ is a surjection the following diagram commutes:
$$
\xymatrix{
\dT[S](\ess)^G \ar[r]^-{\simeq} \ar[d]_{R} & \ess^G \ar@{-->}[r]^-{\simeq} \ar[d]_{f_!} & K({Sets}^G) \ar[d]_{f_!} \\
\dT[S](\ess)^{G'} \ar[r]^-{\simeq} & \ess^{G'} \ar@{-->}[r]^-{\simeq} & K({Sets}^{G'})
}
$$
When $i\colon G \hra G'$ is an injection, the diagram commutes with $(R, f_!)$ replaced by $(F, i^*)$ or by $(V, i_*)$.
\end{prop}
\begin{proof}
It remains only to establish the first equivalence, and the compatibility of that equivalence with $f_!$, $i^*$, and $i_*$.

Without loss of generality we may assume $S$ is finite.  In this case a cofinality argument shows that $\dT[S](\ess)$ is equivalent to $\hocolim_V \Map_*(S^V,S^V)$, where the homotopy colimit occurs over $G$-representations.  The equivalence $\dT[S](\ess)^G \xra{\simeq} \ess^G$ follows and this map obviously respects the restriction, the inclusion of fixed points, and the transfer.
\end{proof} 

\begin{cor}
  There are equivalences 
$$TR^{(n)}(\ess) \simeq \holim_{\mathcal O\subseteq \zpn}K({Sets}^{\zpn/\mathcal O})\simeq \prod_{\mathcal O\subseteq \zpn
} B(\zpn/\mathcal O)_+$$
where the $\mathcal O$'s vary over the open subgroups of $\zpn$.
\end{cor}
\begin{proof}
The first equivalence is immediate from Proposition~\ref{prop-tgkg}, using Lemma~\ref{lem:RFcats2} to express the restriction homotopy limit in terms of open subgroups of $\zpn$.  In section~\ref{sec-ktfinite}, we noted that $K({Sets}^G)\we \bigvee_{K\subseteq G} B(G/K)_+$ for finite abelian $G$.  Moreover we observed that for a surjection $f: G \ra G'$, the restriction map $f_! : K({Sets}^G) \ra K({Sets}^{G'})$ corresponds to the projection 
$\bigvee_{K\subseteq G} B(G/K)_+\to\bigvee_{J\subseteq G'} B(G'/J)_+$
that induces an isomorphism between the $B(G/K)$ and $B(G'/f(K))$ factors whenever $\ker f \subset K$.  The second equivalence in the corollary follows.
\end{proof}

Notice that the homotopy limit here is deceivingly complicated: it is enough to take the homotopy limit over the final subcategory $\dots\subseteq p^{k+1}\zpn\subseteq p^{k}\zpn\subseteq\dots\subseteq \zpn$ corresponding to the restrictions in $TR$ along the inclusions $L_{p^k \id_n}\subseteq L_{p^{k+1} \id_n}$.

%%%%%

\section{$TF^{(n)}$ for the sphere and the Segal conjecture for tori} \label{sec-tfnsphere}

In this section, we calculate the ``topological Frobenius homology" $TF^{(n)}(\ess)$ of the sphere spectrum; recall that this is the homotopy limit over the Frobenius operators acting on the fixed points of higher topological Hochschild homology.  On the one hand, this calculation is closely related to the problem of understanding the higher topological cyclic homology of the sphere---indeed in the Appendix we prove that topological cyclic homology is a homotopy limit of restriction operators on topological Frobenius homology, and we use our computation of the Frobenius homology of the sphere to investigate the diagonal cyclic homology of the sphere.  On the other hand, it is not difficult to see that the function spectrum $F(B\TZ_+, \ess_p)$ is homotopy equivalent to $TF^{(n)}(\ess)_p$.  Our evaluation of topological Frobenius homology therefore gives a precise description of the homotopy type of $F(B\TZ_+, \ess_p)$, that is of the $p$-adic cohomotopy of the classifying space of the torus.  The description of $F(BG_+, \ess_p)$ for finite $G$ is known as the {\em Segal conjecture for $G$}, and was carried out in the 1980s.  Partial results for $G$ compact Lie were obtained, but they focused on the analysis of $\pi_0$ of the function spectrum.  In particular, a complete analysis of the case of a torus was not obtained.  We give such an analysis below.  We hope the result can moreover be used to complete our understanding of the compact Lie group version of the Segal conjecture.

We begin by relating the cohomotopy spectrum of the classifying space of the torus to topological Frobenius homology, and by summarizing our computation of the latter.

\begin{prop} \label{prop-tntf}
There is an equivalence
$$F(B\TZ_+, \ess_p)  \simeq TF^{(n)}(\ess)_p$$
\end{prop}
\begin{proof}
The cohomotopy of the classifying space of the torus is equivalent to a homotopy limit of cohomotopy spectra of classifying spaces of finite groups, which spectra are in turn related to fixed points of equivariant sphere spectra and thereby to Frobenius homology:
$$F(B\TZ_+, \ess_p)  \simeq \holim F(BG_+,\ess_p) 
\simeq \holim \ess_p^G
\simeq \holim T_{\Tnp}(\ess)_p^G 
\simeq (\holim T_{\Tnp}(\ess)^G)_p = TF^{(n)}(\ess)_p$$
Here all the homotopy limits may occur over either the finite subgroups of $\Tnp$ or over the final subcategory of diagonal subgroups $C_{p^l}^{\times n}$.  The first equivalence is established by direct calculation~\cite{carlsson-fpp}, and the second is the Segal conjecture for finite groups~\cite{carlsson-segalconj}.  The third equivalence is Proposition~\ref{prop-tgkg} and the fourth is immediate.
\end{proof}

\begin{theorem} \label{thm-segalconj}
The homotopy groups of the higher topological Frobenius homology of the sphere spectrum are as follows:
$$\pi_*(TF^{(n)}(\ess)_p) = \prod_{k,\alpha} \lim_l \left( \bZ[GL_n(\zp) / \Gamma_{l,k,\alpha}] \otimes \pi_*(\Sigma^{\infty} S^k \sm B \TT^k_+)/p^l \right)$$
Here the product is over $1 \leq k \leq n$ and $\alpha$ is a collection of unordered positive integers $\{n_1, \ldots, n_k\}$.  The limit is over $l \in \NN$, and the group $\Gamma_{l,k,\alpha} \subset GL_n(\zp)$ is determined as follows.  Consider subgroups $K$ of $C_{p^l}^{\times n}$ such that the minimal number of generators of the quotient group $C_{p^l}^{\times n} / K$ is exactly $k$ (we say that $K$ has rank $k$), and the collection of exponents of $p$ in the standard cyclic p-group decomposition of $C_{p^l}^{\times n} / K$ is $\{l - n_1, \ldots, l - n_k\}$ (we say that $K$ has cotype $\alpha$).  The group $GL_n(\zp)$ acts on the set of subgroups $K$ of $C_{p^l}^{\times n}$ with rank $k$ and cotype $\alpha$, and $\Gamma_{l,k,\alpha}$ is the stabilizer of any chosen $K$ under this $GL_n(\zp)$-action.
\end{theorem}

The conceptual origin of this decomposition of the homotopy of $TF^{(n)}(\ess)$ and a more detailed description of the terms involved in the decomposition are given in the following sections.
 
\subsection{The rank filtration of the equivariant sphere spectrum functor} \label{sec-rankfilt}

We express topological Frobenius homology as a homotopy limit of equivariant sphere spectra, and then describe a rank filtration of these spectra.  

By definition $TF^{(n)}(\ess)$ is the homotopy limit $\holim_{Frob_{\cM_n}} T^{(n)}(\ess)^{L_\alpha}$.  In section~\ref{sec-identfrobres}, we saw that the Frobenius indexing category is equivalent to the category $Sub := Sub_{\Tnp}^{\op}$ of finite subgroups of $\Tnp$.  By Proposition~\ref{prop-tgkg}, the fixed points of topological Hochschild homology are equivalent to the fixed points of the equivariant sphere spectrum.  In particular,  topological Frobenius homology can be expressed as $TF^{(n)}(\ess)_p \simeq \holim_{G \in Sub} ({\ess}^G)_p$.  Throughout this section we will abbreviate the functor in this homotopy limit by $\Phi: Sub \ra \SpecCat$; that is $\Phi(G) = ({\ess}^G)_p$ and $\Phi(i:H \hra G)=i^*$.  Moreover, in light of the results of section~\ref{sec-kthyequivsphere} and by abuse of notation, we will not distinguish between the equivariant sphere spectrum ${\ess}^G$ and the $K$-theory of $G$-sets $K({Sets}^G)$---indeed most of our analysis will occur in the world of $G$-sets---and we will generally let p-completion be implicit.

The rank filtration of the equivariant sphere spectrum functor $\Phi$ is obtained by filtering the symmetric monoidal category ${Sets}^G$.    For any transitive $G$-set $X \cong G/K$, we define the rank of $X$ to be the minimal number of generators required to generate $G/K$.  We then define ${\mathcal C}^G[k]$ to be the full subcategory of ${Sets}^G$ of those $G$-sets all of whose orbits have rank less than or equal to $k$.  These subcategories have the following properties. 

\begin{itemize}
\item{${\mathcal C}^G[0]$ is equivalent to the category of finite sets with trivial $G$-action, hence is equivalent to the sphere spectrum. ${\mathcal C}^G[n]$ is all of ${{Sets}}^G$. }
\item{${\mathcal C}^G  [k]$ is a symmetric monoidal subcategory of ${{Sets}}^G$, whose associated spectrum we will denote by ${\ess}^G[k]$.  We have an increasing sequence of spectra 
$$   {\ess}^G[0] \subseteq {\ess}^G[1] \subseteq \ldots \subseteq {\ess}^G[n] = {\ess}^G
$$}
\item{The subcategories ${\mathcal C}^G[k]$ are preserved under the map $\Phi(i)$, where $i: H \hra G$, and so they will create their own spectrum-valued diagrams $\Phi [k]$, with $\Phi[k](G) := K({\mathcal C}^G[k])$.  }
\end{itemize}

 We now have a filtration of $\holim_{Sub}{\Phi}
$ by 
$$ \holim_{Sub}{\Phi [0]} \subseteq  \holim_{Sub}{\Phi [1]} \subseteq \ldots \subseteq  \holim_{Sub}{\Phi [n]} =  \holim_{Sub}{\Phi }
$$
The  relative terms  in this filtration are the spectra 
$\holim_{Sub}{\Phi [k] /\Phi [k-1]} $, in view of the following general result about spectrum homotopy limits.  

\begin{prop}
Let $\cC$ denote a small category, and suppose that we have two spectrum-valued functors $F$ and $G$ on $\cC$, together with a natural transformation $\varphi: F \rightarrow G$.  Let $C(\varphi)$ denote the functor with $C(\varphi) (x) $ equal to the mapping cone of $F(x) \rightarrow G(x)$.  There is a cofibration sequence of spectra 
$
\holim_{\cC}{F} \rightarrow \holim_{\cC}{G} \rightarrow 
\holim_{\cC}{C(\varphi)}
$.
\end{prop} 

Our next task is to evaluate the subquotients $\holim_{Sub} {\Phi [k]/\Phi [k-1]}$ of the rank filtration.  Toward that end, we express the quotients $\ess^G[k]/\ess^G[k-1]$ themselves as $K$-theory spectra of symmetric monoidal categories.  Define the symmetric monoidal category ${\mathcal C}^G {\langle}k{\rangle}$  to be the category of finite $G$-sets all of whose orbits have rank equal to $k$.  Given an inclusion of abelian groups $i: H \hookrightarrow G$, define a functor $i^*: {\mathcal C}^G {\langle}k{\rangle} \rightarrow {\mathcal C}^H {\langle}k{\rangle}$ as follows.  For a finite $G$-set $X$, all of whose orbits have rank $k$, decompose $X$ as $X = X^{\prime} \coprod X^{\prime \prime}$, where $X^{\prime}$ is the union of all $H$-orbits of $X$ that have rank $k$.  The functor $i^*$ takes $X$ to $X^{\prime}$.  Altogether the construction ${\mathcal C}^G \langle k \rangle$ gives a contravariant functor ${\mathcal I}$ from the category of finite subgroups of $\TZ$ to the category of small categories.  
\begin{prop}
The quotient ${\ess}^G[k]/{\ess}^G[k-1]$ is the $K$-theory spectrum associated to the symmetric monoidal category ${\mathcal C}^G{\langle}k{\rangle}$  of finite $G$-sets whose orbits all have rank $k$.  For $i: H \hra G$ an inclusion of abelian groups, the map $K({\mathcal I}(i^*)): {\ess}^G[k]/{\ess}^G[k-1] \rightarrow {\ess}^H[k]/{\ess}^H[k-1]$ is induced by the natural restriction map $i^*: \ess^G[k] \ra \ess^H[k]$.
\end{prop}
\nid From now on, we will write ${\ess}^G{\langle}k{\rangle}$ for the subquotient ${\ess}^G[k]/{\ess}^G[k-1]$.

\subsection{The cotype decomposition of the sphere and a splitting of the rank filtration} \label{sec-cotype}

We can simplify the homotopy limit in topological Frobenius homology as follows.  The full subcategory $\cD_n$ of $Sub$ on the diagonal subgroups $C^{\times n}_{p^l}$ of $\Tnp$ is final, and so $TF^{(n)}(\ess) \simeq \holim_{l \in \cD_n} \Phi |_{\cD_n}$.  Restricting to the subcategory $\cD_n$ will allow us to produce a decomposition of the subquotients of the rank filtration, which in turn will force the rank filtration to split: $ \holim_{Sub}{\Phi} \simeq \bigvee _{k=0} ^n \holim_{Sub} {\Phi [k]/\Phi [k-1]}
$.

For any subgroup $K \subseteq C_{p^l}^{\times n} \subset \TT^{n}$, we will let the {\em type} of $K$ denote the collection of exponents of $p$  occurring in the decomposition of $C_{p^l}^{\times n}/K$ into a direct sum of finite cyclic $p$-groups.  That is, if the type of $K$ is the set $\{ e_1, e_2, \ldots , e_t \}$, then we have 
$$ C_{p^l}^{\times n}/K \cong \bigoplus _{i=1}^t C_{p^{e_i}}
$$
We observe that (i) $1 \leq e_i \leq l $ and (ii) $t \leq n$. By the {\em cotype} of $K$, we mean the collection $\{ l-e_1, l -e_2, \ldots, l - e_t \}$, and we denote it by $ct_l(K)$,  to emphasize that it depends on $l$.  We will need the following result regarding this invariant. 

\begin{prop}\label{prop-rankcotype}
Let $K \subseteq C_{p^l}^{\times n}$, and suppose that 
$$rk(C_{p^l}^{\times n} /K) = rk(C_{p^{l-1}}^{\times n}  /K\cap C_{p^{l-1}}^{\times n})$$ Then $$ct_l(K) = ct_{l-1}(K \cap C_{p^{l-1}}^{\times n})$$
\end{prop} 
\begin{proof}
We first reinterpret $ct_l(K)$. Since $K$ is a subgroup of $C_{p^l}^{\times n}$, it can be decomposed as
$$ K \cong (C_{p^l})^{s} \oplus \bigoplus _{i= 1}^{n-s}  C_{p^{f_i}}
$$
where $0 \leq f_i \leq l-1$.  The numbers $s$ and the collection of numbers $f_i$ are unique, up to a possible reordering of the $f_i$'s.  The corresponding decomposition for $K \cap C_{p^{l-1}}^{\times n}$ is of the form 
$$ K \cap  C_{p^{l-1}}^{\times n} \cong C_{p^l}^{s + t} \oplus \bigoplus _{i= 1}^{n-s - t} C_{p^{f_i}}
$$
where $t$ is the number of values of $i$ for which $f_i = l-1$.  It is clear from the definitions that the rank of $C_{p^l}^{\times n}/K$ is $n- s$, and that the rank of $C_{p^{l-1}}^{\times n} /K \cap C_{p^{l-1}}^{\times n}  $ is $n- s - t$.  Since we are assuming that  $rk(C_{p^l}^{\times n} /K) = rk(C_{p^{l-1}}^{\times n}  /K\cap C_{p^{l-1}}^{\times n} )$, it follows that in our case $t = 0$, so in fact we have $f_i \leq l-2$ for all $i$.  Finally, we now have $$ct_l(C_{p^l}^{\times n} /K) = \{ f_1, f_2, \ldots , f_{n- s} \} = 
ct_{l-1}(C_{p^{l-1}}^{\times n}  /K\cap C_{p^{l-1}}^{\times n} ) $$ as required. \end{proof}

  We next consider the symmetric monoidal category ${\mathcal C}^{C_{p^l}^{\times n}}{\langle} k{\rangle}$ whose objects are the finite  $C_{p^l}^{\times n}$-sets all of whose orbits have rank $k$.  Note that $k \leq n$.  Any orbit in a finite $C_{p^l}^{\times n}$-set has the form $C_{p^l}^{\times n} / K$, where $K$ is a subgroup.  By the {\em cotype} of the orbit, we will mean $ct_l(K)$.  Any finite $C_{p^l}^{\times n}$-set $X$ whose orbits all have rank $k$  has a canonical decomposition
$$ X \cong  \coprod _{\alpha} X_{\alpha}
$$
Here $X_{\alpha}$ denotes the union of all orbits whose cotype is equal to $\alpha$.  Further, $\alpha$ ranges over unordered families $\{ n_1, n_2, \ldots , n_k \}$, where the $n_i$'s are positive integers.  It follows that the symmetric monoidal category 
${\mathcal C}^{C_{p^l}^{\times n}}{\langle} k{\rangle}$ has a decomposition
$$ {\mathcal C}^{C_{p^l}^{\times n}}{\langle} k{\rangle}\cong \prod _{\alpha} {\mathcal C}^{C_{p^l}^{\times n}}{\langle}k, \alpha {\rangle}
$$
where ${\mathcal C}^{C_{p^l}^{\times n}}{\langle} k, \alpha {\rangle}$ is the symmetric monoidal category of finite $C_{p^l}^{\times n}$-sets all of whose orbits have rank $k$ and all of whose cotypes are $\alpha$.  
We will write ${\ess}^{C_{p^l}^{\times n}}{\langle}k, \alpha {\rangle}$ for the corresponding spectrum; we now have a decomposition of spectra  
$$  {\ess}^{C_{p^l}^{\times n}}{\langle}k {\rangle}  \simeq 
\bigvee_{\alpha} {\ess}^{C_{p^l}^{\times n}}{\langle}k,\alpha {\rangle}
$$

\begin{prop}The restriction maps ${\ess}^{C_{p^l}^{\times n}}{\langle}k{\rangle} \rightarrow {\ess}^{C_{p^{l-1}}^{\times n} }{\langle}k{\rangle} $ respect the cotype decomposition.  
\end{prop}
\nid This result is a consequence of Proposition \ref{prop-rankcotype}: as long as the ranks stay constant, so do the types.  

 We can define functors $\Phi [k,\alpha]$ on ${\mathcal D}_n$ by 
$\Phi [k,\alpha](C_{p^l}^{\times n}) = {\ess}^{C_{p^l}^{\times n}}\hspace{-.2cm}{\langle}k,\alpha {\rangle}$.  Observe that for any fixed $l$, the set of possible cotypes for $C_{p^l}^{\times n}$-orbits is finite, since the integers $n_i$ involved must be less than or equal to $l$.  Consequently, we have a decomposition of $\Phi[k]/\Phi [k-1]$ into a {\em product} of functors $\Phi [k,\alpha]$.  

\begin{cor}\label{cor-decomp} There are equivalences
$$\holim_{Sub}{\Phi [k] }/ \holim_{Sub}{\Phi [k-1] } \simeq \holim_{Sub} \Phi[k]/\Phi[k-1] \simeq
\holim_{{\mathcal D}_n} \Phi[k]/\Phi[k-1] \simeq \prod _{\alpha} \holim_{{\mathcal D}_n}{\Phi [k, \alpha]}$$
\end{cor} 

We can bootstrap this decomposition into a splitting of the rank filtration:

\begin{theorem} \label{thm-ranksplit}
  The projection
  $$\holim_{l \in {\mathcal D}_n}\Phi[k](l)\to\holim_{l \in {\mathcal D}_n}\Phi[k](l)/\Phi[k-1](l)$$
  splits in the homotopy category.  (Here $\Phi[k](l)$ is shorthand for $\Phi[k](C_{p^l}^{\times n})$.)  More precisely, the diagram
$$
  \xymatrix{{\holim_l \prod_\alpha\Phi[k,\alpha](l)}
    \ar[r]\ar[dr]^{\simeq}&
  {\holim_l \Phi[k](l)}\ar[d]\\
{\prod_\alpha\holim_l \Phi[k,\alpha](l)}\ar[u]^{\cong}&
{\holim_l \Phi[k](l)/\Phi[k-1](l)}
}
$$
commutes, where the horizontal map is induced by inclusion of
categories, the diagonal map is induced from the equivalence of Corollary~\ref{cor-decomp},
the left vertical map is the canonical isomorphism, and the right
vertical map is the projection.
\end{theorem}
\begin{proof}
When $l$ is fixed, the upper triangle is exactly the identification of the filtration quotient \mbox{$\Phi[k](l) / \Phi[k-1](l)$} with the finite product over the cotypes.  Considering the homotopy limit over $l$, Corollary~\ref{cor-decomp} shows that the diagonal map becomes an equivalence.
\end{proof}
\begin{cor}\label{cor-decompofinal}
  There is an equivalence
$$TF^{(n)}(\ess)_p\simeq \holim_{Sub}\Phi\simeq\prod_{k,\alpha} \holim_{{\mathcal D}_n} \Phi[k,\alpha]$$
\end{cor}
\begin{proof}
The first equivalence was mentioned at the beginning of section~\ref{sec-rankfilt}.  The second is a combination of Theorem~\ref{thm-ranksplit}, which splits the rank filtration into its subquotients, and Corollary~\ref{cor-decomp}, which decomposes the subquotients by cotype.
\end{proof}

\subsection{The homotopy type of the fixed rank-cotype components of $TF^{(n)}$} \label{sec-rankcotype}

The next stage is to evaluate the homotopy of the individual rank-cotype components $\holim_{{\mathcal D}_n}{\Phi [k, \alpha]}$.  

Before directly confronting the rank-cotype computation, we record a few necessary generalities about group actions on spectra.  Consider a group $G$, a subgroup $K \subseteq G$, and a basepoint-preserving action of $K$ on a based space $X$.  The based space $G_+ \smsh_{K} X$ is the quotient of $G_+ \wedge X$ by the equivalence relation $gk \wedge x \simeq g \wedge kx$.  We extend this notion to spectra, for a naive action of $K$ on a spectrum $X$, by letting $G_+ \smsh_{K}X$ be the spectrum associated to the prespectrum $G_+ \smsh_{K}X(-)$.  This construction has the following universal property. 

\begin{prop} \label{universality} Given a spectrum $Y$ with $G$-action, a subgroup $K \subseteq G$, a spectrum $X$ with \mbox{$K$-action}, and a $K$-equivariant map $\psi : X \rightarrow Y$, there is a unique $G$-equivariant extension $\overline{\psi}: G _+ \smsh_{K} X \rightarrow Y$ that agrees with $\psi $ on $X = K_+ \smsh_{K} X \subseteq G _+ \smsh_{K} X $.  
\end{prop}

 The non-equivariant homotopy type of $G _+ \smsh_{K} X $ can be described
 as follows.
\begin{lemma}\label{untwist}Let $K$ be a subgroup of $G$, and
 let $X$ be a spectrum with a $K$-action.  Let $\{ g_{\alpha} \}_{\alpha \in A}$ denote a set of coset representatives for $G/K$, so that the cosets are given by $\{ [ g_{\alpha} ] \} _{\alpha \in A}$.  The  map 
$$ \theta: G/K_+ \wedge X \rightarrow G_+ \wedge _K X,
$$
defined levelwise by $[g_\alpha ] \wedge x  \mapsto g_{\alpha } \wedge_K x$, is an equivalence of spectra. 
\end{lemma} 

The naturality properties of this construction are as follows.  Suppose we have a group $G$, subgroups $K^{\prime} \subseteq K \subseteq G$, group actions of $K$ and $K^{\prime} $ on spectra $X$ and $X^{\prime}$ respectively, and an equivariant map $f : X^{\prime} \rightarrow X$.  There is a naturally defined map 
$ \eta : G \wedge _{K^{\prime} }X^{\prime} \rightarrow G \wedge _K X$
induced by the levelwise map $\eta (g \wedge _{K^{\prime}} x^{\prime}) = g \wedge _K f(x^{\prime})
$.  Choose families of coset representatives $ \{ g_{\alpha } \}_{\alpha \in A}$ and $\{ g^{\prime} _{\beta} \} _{\beta \in B}$ for $G/K$ and $G / K^{\prime}$ respectively.  The coset $[g^{\prime}_{\beta}]$ determines a $K$-coset, which can  be written in the form $[g _{\alpha}]$ for a unique value of $\alpha$.  This means that there is a distinguished element $k_{\beta} \in K$ so that $g^{\prime}_{\beta} = g_{\alpha} k_{\beta}$.  

\begin{lemma}\label{mapbehavior} The diagram 
$$
  \begin{CD}
    G/K'_+\smsh X'@>{\theta}>> G\smsh_{K'}X'\\
    @V{\lambda}VV@V{\eta}VV\\
    G/K_+\smsh X@>{\theta}>> G\smsh_{K}X
  \end{CD}
$$
commutes, where $\lambda$ is the map induced by the levelwise equation 
$\lambda ( [g_{\beta}^{\prime}] \wedge x^{\prime} ) = [g_{\alpha}] \wedge k_{\beta} f(x^{\prime} )$---here $\alpha$ is chosen so that the $K$-coset $[g^{\prime}_{\beta}]$ is equal to $[g_{\alpha}]$.  
\end{lemma} 

We now introduce some terminology concerning the collection of groups of a fixed cotype.  For each $k$ and $\alpha$, let ${\mathcal M}[k,\alpha](l)$ denote the based set obtained by adjoining a disjoint basepoint to the set of subgroups $K \subseteq C_{p^l}^{\times n}$ such that $C_{p^l}^{\times n} /K$  has rank $k$ and $K$ has cotype $\alpha$.  We note that there are maps 
$\theta(l, l^{\prime}): {\mathcal M}[k,\alpha](l) \rightarrow {\mathcal M}[k,\alpha](l^{\prime})$ whenever $l \geq l^{\prime}$, defined by $\theta (l, l^{\prime})(K) = K \cap C_{p^{l^{\prime}}}^{\times n}$
when $C_{p^{l^{\prime}}}^{\times n}/K \cap C_{p^{l^{\prime}}}^{\times n}$ has rank $k$ and $K \cap C_{p^{l^{\prime}}}^{\times n}$ has cotype $\alpha$, and $\theta (l, l^{\prime})(K) = *$ otherwise.    We now   describe this set ${\mathcal M}[k,\alpha](l)$ as a quotient of $GL_n (\zp)$.  
\begin{lemma}\label{Mformoduli}  The group $GL_n ({\bZ}_p)$ acts on the based set ${\mathcal M}[k,\alpha](l)$.  The action is transitive on non-basepoint elements, and the stabilizer of each non-basepoint element is a finite index subgroup. Therefore, ${\mathcal M}[k,\alpha](l)$ can be described  as $(GL_n({\bZ}_p)/\Gamma _l)_+$, where $\Gamma _l$ is the stabilizer of an element in ${\mathcal M}[k,\alpha](l)$.   
\end{lemma}
\begin{proof} Since fixing rank and cotype fixes the abstract
  isomorphism type of the subgroup, this is Corollary \ref{tranact}. 
\end{proof}

 We proceed to analyze $ K({\mathcal C}^{C_{p^l}^{\times n}}{\langle} k, \alpha {\rangle}) $, where as before ${\mathcal C}^{C_{p^l}^{\times n}}{\langle} k, \alpha {\rangle}$ is the symmetric monoidal category of finite $C_{p^l}^{\times n}$-sets all of whose orbits have rank $k$ and all of whose cotypes are $\alpha$.  Choose any subgroup $K \subseteq C_{p^l}^{\times n}$ for which the orbit $ C_{p^l}^{\times n}/K$ has rank $k$ and such that $K$ has cotype $\alpha$.  Let $\Gamma$ denote the stabilizer of $K$ under the action of $GL_n({\bZ}_p)$ on ${\mathcal M}[k,\alpha](l)$.   Let 
$${\mathcal C}^{C_{p^l}^{\times n}}{\langle} K {\rangle} \subseteq
 {\mathcal C}^{C_{p^l}^{\times n}}{\langle} k, \alpha {\rangle} $$
denote the symmetric monoidal subcategory on those $C_{p^l}^{\times n}$-sets all of whose points have $K$ as their stabilizer.  On the one hand, the group $\Gamma$ clearly acts on this category ${\mathcal C}^{C_{p^l}^{\times n}} \langle K \rangle$ via its action on the quotient group $C_{p^l}^{\times n} /K$, and consequently acts on the corresponding spectrum.  On the other hand, the full group $GL_n({\bZ}_p)$ acts on the category 
$ {\mathcal C}^{C_{p^l}^{\times n}}{\langle} k, \alpha {\rangle}$ because it acts on the group $C_{p^l}^{\times n}$.  The inclusion $ {\mathcal C}^{C_{p^l}^{\times n}}{\langle} K {\rangle}  \subseteq {\mathcal C}^{C_{p^l}^{\times n}}{\langle} k, \alpha {\rangle}$ is clearly $\Gamma$-equivariant, so Proposition \ref{universality} yields a map of spectra
$ \rho : GL_n({\bZ}_p)_+ \smsh_{\Gamma} K({\mathcal C}^{C_{p^l}^{\times n}}{\langle} K {\rangle}) \rightarrow K({\mathcal C}^{C_{p^l}^{\times n}}{\langle} k, \alpha {\rangle}).
$

\begin{prop} \label{prop-glngamma}
The map $\rho: GL_n({\bZ}_p)_+ \smsh_{\Gamma} K({\mathcal C}^{C_{p^l}^{\times n}}{\langle} K {\rangle}) \rightarrow K({\mathcal C}^{C_{p^l}^{\times n}}{\langle} k, \alpha {\rangle})$ is an equivalence.  
\end{prop}
\begin{proof} 
Using Lemma \ref{untwist} note that the map $\rho$ has source
  $$GL_n({\bZ}_p)_+ \smsh_{\Gamma} K({\mathcal C}^{C_{p^l}^{\times n}}{\langle} K {\rangle}) \simeq
  (GL_n({\bZ}_p)/\Gamma)_+ \smsh K({\mathcal C}^{C_{p^l}^{\times
      n}}{\langle} K {\rangle}) \simeq
      \bigvee_{\bar{g}\in GL_n({\bZ}_p)/\Gamma} K({\mathcal C}^{C_{p^l}^{\times
      n}}{\langle} K {\rangle})$$  
The target of the map $\rho$ is
$$K\left(\cC^{C_{p^l}^{\times n}}\langle k,\alpha\rangle\right) \simeq
K\left(\prod_{L\in\mathcal M[k,\alpha](l)}
\cC^{C_{p^l}^{\times n}}\langle L\rangle\right)$$
By Lemma \ref{Mformoduli}, we have a preferred bijection 
$$
  \begin{CD}
    (GL_n({\bZ}_p)/\Gamma)_+@>{\bar{g}\mapsto \bar{g}\cdot K}>> \mathcal
  M[k,\alpha](l)
  \end{CD}
 $$ 
Under these equivalences
the map $\rho$ takes the $\bar{g}$-summand in
the wedge to the $(\bar{g}\cdot K)$-th factor in the $K$-theory product---the map is induced by the
isomorphism
${\mathcal C}^{C_{p^l}^{\times n}}{\langle} K {\rangle}\cong 
{\mathcal C}^{C_{p^l}^{\times n}}{\langle} \bar{g}\cdot K {\rangle}
$
given by $\bar{g}$.
Because $\Gamma$ is of finite index the map from the wedge to
the product is an equivalence.
\end{proof} 

Proposition~\ref{prop-glngamma} establishes the homotopy type of the values $\Phi[k,\alpha](l) = K({\mathcal C}^{C_{p^l}^{\times n}}\langle k,\alpha \rangle)$ of the functor $\Phi[k,\alpha]$.  We now describe the maps between the various spectra $\Phi[k,\alpha](l)$. We begin by choosing for each $l$ a subgroup $K_l \subseteq C_{p^l}^{\times n}$, with $K_l \in {\mathcal M}[k,\alpha ](l)$, and $K_l \cap C_{p^l}^{\times n}  = K_{l^{\prime}}$ whenever $l \geq l^{\prime}$.  For example, decompose $C_{p^l}^{\times n} $ as $C_{p^l}^{\times n-k } \times  C_{p^l}^{\times k} $, and select a subgroup $E \subseteq C_{p^l}^{\times k}$ whose family of cyclic factors is $\alpha$.  Then the groups $K_{l^{\prime}} = C_{p^{l^{\prime}}}^{\times n-k} \times E$ form an appropriate family of subgroups.  We define $\Gamma _l $ to be the stabilizer of $K_l$ under the $GL_n ({\bZ}_p)$-action.  It is clear that $\Gamma_{l+1} \subseteq \Gamma _l$.  We now define $$\tau _l: {\mathcal C}^{C_{p^l}^{\times n}}{\langle} K_l{\rangle} \rightarrow 
{\mathcal C}^{C_{p^{l-1}}^{\times n} }{\langle} K_{l-1} {\rangle}
$$
to be the functor obtained by restricting $C_{p^l}^{\times n}/K_l$-sets along the inclusion 
$C_{p^{l-1}}^{\times n}  /K_{l-1} \hookrightarrow  C_{p^{l}}^{\times n}/K_{l}$.  This functor is clearly $K_{l}$-equivariant, and an application of the universal property in Proposition \ref{universality} yields a map of spectra 
$$t_l: GL_n({\bZ}_p) \smsh_{\Gamma_l} K({\mathcal C}^{C_{p^l}^{\times n}}{\langle} K_l{\rangle} ) \rightarrow 
GL_n({\bZ}_p) \smsh_{\Gamma_{l-1}} K({\mathcal C}^{C_{p^{l-1}}^{\times n} }{\langle} K_{l-1}{\rangle} )
$$

\begin{prop} \label{main} The functor $\Phi [k,\alpha ]$ is naturally
  equivalent to the functor $\Psi$  given by $\Psi(C_{p^l}^{\times n})
  = GL_n({\bZ}_p) \smsh_{\Gamma_l} K({\mathcal C}^{C_{p^l}^{\times
      n}}{\langle} K_l{\rangle} )$, and taking $C_{p^{l-1}}^{\times n} \hookrightarrow C_{p^l}^{\times n}$ to the map
  $t_l: \Psi (C_{p^l}^{\times n}) \ra \Psi (C_{p^{l-1}}^{\times
    n} )$.  In other words, the diagram
$$
  \begin{CD}
    GL_n({\bZ}_p) \smsh_{\Gamma_{l}} 
    K({\mathcal C}^{C_{p^l}^{\times n}}{\langle} K_l{\rangle} )
    @>{\rho}>{\simeq}> 
    K({\mathcal C}^{C_{p^l}^{\times n}}{\langle} k, \alpha {\rangle})
    @=\Phi[k,\alpha](l)\\
    @V{t_l}VV@.@V{\Phi[k,\alpha](``l-1 \hookrightarrow l")}VV\\
    GL_n({\bZ}_p) \smsh_{\Gamma_{l-1}} K({\mathcal C}^{C_{p^{l-1}}^{\times
      n}}{\langle} K_{l-1}{\rangle}) @>{\rho}>{\simeq}>
    K({\mathcal C}^{C_{p^{l-1}}^{\times n}}{\langle} k, \alpha {\rangle})
    @=\Phi[k,\alpha](l-1)
  \end{CD}
$$
commutes.
\end{prop}

Next we would like to describe the homotopy groups of $\holim_{{\mathcal D}_n} \Phi[k,\alpha]$ in terms of the homotopy limit of the factors $K(\cC^{C_{p^l}^{\times n}} \langle K_l \rangle )$.  The general situation we are in is as follows.  Consider a profinite group $G$ and a sequence of finite index subgroups $\Gamma_l \subseteq G$, $l \geq 0$, with $\Gamma_{l+1} \subseteq \Gamma_l$.  Let $\cdots \ra X_l \ra X_{l-1} \ra \cdots \ra X_0$ be an inverse system of spectra, and suppose for all $l$ the spectrum $X_l$ has a $\Gamma_l$-action such that the map $X_l \ra X_{l-1}$ is equivariant.  Proposition~\ref{universality} provides maps of spectra
$$  \cdots \rightarrow G_+ \smsh_{\Gamma_l} X_l \rightarrow G_+ \smsh_{\Gamma_{l-1}} X_{l-1} \rightarrow \cdots \rightarrow  G_+ \smsh_{\Gamma_0} X_0
$$
and by Lemma~\ref{untwist} and Lemma~\ref{mapbehavior} this inverse system is equivalent to the system
$$ \cdots \to{G/\Gamma_l}_+ \wedge X_l \rightarrow {G/\Gamma_{l-1}}_+ \wedge X_{l-1}
\rightarrow \cdots \to {G/\Gamma_0}_+ \wedge X_0
$$
We have equivalences
$$\holim_l {G/\Gamma_l}_+ \sm X_l \simeq \holim_{l,k} {G/\Gamma_l}_+ \sm X_k
\simeq \holim_l {G/\Gamma_l}_+ \sm \holim_k X_k$$
Let $X$ denote $\holim_k X_k$.  We now need only understand the homotopy groups of homotopy limits of inverse systems of the form ${G/\Gamma_l}_+ \sm X$.
\begin{prop} \label{prop-profacts}
If $\pi_t(X)$ is a finitely generated $\bZ_p$-module, then
$$\pi_t \holim_l ({G/\Gamma_l}_+ \sm X) \cong \lim_l (\bZ[G/\Gamma_l] \otimes \pi_t(X)/p^l)$$
\end{prop}
\begin{proof}
The systems $\pi_t({G/\Gamma_l}_+ \sm X)$ are Mittag-Leffler, as the maps are surjective, and thus the corresponding $\lim^1$ terms vanish.  Thus $\pi_t \holim_l {G/\Gamma_l}_+ \sm X \cong \lim_l \bZ[G/\Gamma_l] \otimes \pi_t(X)$.  Because $\bZ[G/\Gamma_l]$ is free and finitely generated, provided $\pi_t(X)$ is a finitely generated $\bZ_p$-module, it is moreover the case that $\lim_l \bZ[G/\Gamma_l] \otimes \pi_t(X) \cong \lim_l \bZ[G/\Gamma_l] \otimes \lim_k \pi_t(X)/p^k \cong \lim_l \bZ[G/\Gamma_l] \otimes \pi_t(X)/p^l$.
\end{proof}

 We now focus on our particular situation: take $G$ to be $GL_n({\bZ}_p)$, and $\Gamma_l$ to be, as before, the stabilizer of an element of ${\mathcal M}[k,\alpha ](l)$.  Let $X_l$ be $K({\mathcal C}^{C_{p^l}^{\times n}}{\langle} K_l  {\rangle})_p$, and let the maps in the inverse system be restrictions along the inclusions $C_{p^l}^{\times n}/ K_l \hookrightarrow C_{p^{l+1}}^{\times n}/K_{l+1}$. From the definitions we have an equivalence
$$K({\mathcal C}^{C_{p^l}^{\times n}}{\langle} K_l  {\rangle}) \simeq 
\Sigma^{\infty} \!\left(BC_{p^{l-n_1}} \times BC_{p^{l-n_2}} \times \cdots \times BC_{p^{l-n_k}}\right)_+$$
where $\alpha $ is the family $\{ n_1, n_2, \ldots n_k \}$, and the maps in the inverse system are the transfers obtained from the covering spaces 
$$BC_{p^{l-n_1}} \times BC_{p^{l-n_2}} \times \cdots \times BC_{p^{l-n_k}} \rightarrow BC_{p^{l+1-n_1}} \times BC_{p^{l +1 -n_2}} \times \cdots \times BC_{p^{l+1 -n_k}}
$$ 
For each factor  $BC_{p^{l-n_i}} $, we have the $S^1$-transfer $\Sigma ^{\infty}S^1 \wedge B\mathbb{T}_+ \rightarrow \Sigma^{\infty} (BC_{p^{l-n_i}})_+$.  We may smash $k$ such maps together to obtain a map 
$ \Sigma^{\infty} S^k \wedge B\mathbb{T}^k _+ \rightarrow \Sigma
^{\infty} 
(
BC_{p^{l-n_1}} \times BC_{p^{l-n_2}} \times \cdots \times
BC_{p^{l-n_k}}
)_+
$.
Standard compatibility formulae for the transfer show that these maps fit together to yield a map 
$$ \Sigma^{\infty} S^k \wedge B\mathbb{T}^k _+  \rightarrow
\holim_l {\Sigma^{\infty} 
\!\left(
BC_{p^{l-n_1}} \times BC_{p^{l-n_2}} \times \cdots \times
BC_{p^{l-n_k}}
\right)_+}
$$
After $p$-completion, this map becomes an equivalence, as can be checked by homological calculations of the transfer maps.  The homotopy groups of $(\Sigma^{\infty} S^k \wedge B\mathbb{T}^k_+)_p$ are evidently finitely generated $\bZ_p$-modules.

 The following is now a consequence of Proposition \ref{prop-profacts}
 and Corollary \ref{cor-decompofinal}.

\begin{theorem} \label{computation} The homotopy groups of the fixed rank-cotype piece of the topological Frobenius homology of the sphere are
$$\pi _* (\holim_{{\mathcal D}_n}{\Phi [k, \alpha ] })\cong 
\lim_l \left(\bZ[GL_n({\bZ}_p)/\Gamma_l] \otimes \pi _* (\Sigma ^{\infty} S^k \wedge B\mathbb{T}^k_+)/p^l\right)$$
In concrete terms, this means that for every free generator in $\pi _t (\Sigma ^{\infty} S^k \wedge B\mathbb{T}^k_+)$, there is a summand of the form 
$\lim_{i,j}  ({\bZ}/p^i) [GL_n({\bZ}_p)/\Gamma_j]
$
in $\pi _t (\holim_{{\mathcal D}_n}{\Phi [k, \alpha ] })$, and for every finite cyclic summand $C$ in $\pi _t(\Sigma ^{\infty} S^k \wedge B\mathbb{T}^k_+)$, there is  a summand of the form 
$C \otimes \lim_l {\bZ}[GL_n({\bZ}_p)/\Gamma_l]$. 

The homotopy groups of the topological Frobenius homology $TF^{(n)}(\ess)_p$, and therefore the cohomotopy groups of the classifying space of the torus, are given as the product
of these terms as $k$ and $\alpha$ vary over the rank and cotype.
\end{theorem}

\appendix

\renewcommand{\thetheorem}{A.\arabic{theorem}}
\setcounter{theorem}{0}
\section*{Appendix.  Cyclic homology as a homotopy limit of Frobenius homology}

Ordinary topological cyclic homology can be expressed in term of topological Frobenius homology:
$$TC(A) = \holim \left( TF(A) \overset{R}{\underset{\id}{\rightrightarrows}} TF(A) \right)$$
Note that this homotopy limit is equal to the homotopy limit of the diagram with a single object $TF(A)$ and with morphisms the self-maps $R^k: TF(A) \ra TF(A)$ for $0 \leq k < \infty$.

The purpose of this appendix is two-fold.  First, we show that higher topological cyclic homology is the homotopy equalizer of the action of all restriction maps on higher topological Frobenius homology:
\begin{prop} \label{prop-tctf}
There is an equivalence
$TC^{(n)}(A) \simeq \holim_{\mathcal{M}_n} TF^{(n)}(A)$
where $\mathcal{M}_n$ is, as before, the monoid of isogenies of the torus, and where the action of the elements $m \in \mathcal{M}_n$ on $TF^{(n)}(A)$ is induced by the restriction operators $R_{m}$.
\end{prop}
\nid The proposition shows that the computation in section~\ref{sec-tfnsphere} of $TF^{(n)}(\ess)$ is an essential ingredient in any investigation of higher topological cyclic homology.  

The second purpose of this appendix is to illustrate the power of expressing cyclic homology in terms of Frobenius homology by applying the calculation of $TF^{(n)}(\ess)$ to describe the filtration quotients of the rank filtration of the diagonal cyclic homology of the sphere:
\begin{theorem}
The restriction maps corresponding to diagonal isogenies (those of the form $p^l \cdot \id_n$) preserve the rank filtration of the Frobenius homology of the sphere $TF^{(n)}(\ess)$.  The rank filtration therefore descends to diagonal cyclic homology $TC^{\Delta}(\ess)$, and the homotopy groups of the filtration quotients of the diagonal cyclic homology of the sphere are
$$\pi_*((TC^{\Delta}(\ess)[k]/TC^{\Delta}(\ess)[k-1])_p) \cong \prod_{\alpha \in \mathcal{O}_k} 
\lim_l \left( \bZ[GL_n(\zp) / \Gamma_{l,k,\alpha}] \otimes \pi_*(\Sigma^{\infty} S^k \sm B \TT^k_+)/p^l \right)$$
Here $\mathcal{O}_k$ is the set of unordered collections of positive integers $\{e_1, e_2, \ldots, e_k\}$ such that at least one $e_i$ is equal to 1; the group $\Gamma_{l,k,\alpha}$ is, as before, the stabilizer of a subgroup of $C^{\times n}_{p^l}$ of rank $k$ and cotype $\alpha$.
\end{theorem}
\nid The attaching map of the filtration of diagonal cyclic homology is non-trivial even in the one-dimensional case, and it remains an open problem to determine the attaching maps of the filtration for higher diagonal cyclic homology.

\vspace{10pt}
We begin by developing a slight generalization of Proposition~\ref{prop-tctf}.
Recall that the twisted arrow category $\Ar_{\mathcal{M}}$ of a monoid $\mathcal{M}$, where $\mathcal{M}$ is viewed as a category with one object $\mu$, has objects the elements $\mu \xra{m} \mu$ of the monoid, and morphisms $({m_4}^*,{m_2}_*): m_1 \ra m_3$ given by diagrams
$$
  \begin{CD}
    \mu @>{m_2}>> \mu \\ @V{m_1}VV @V{m_3}VV \\ \mu @<{m_4}<< \mu
  \end{CD}
$$
Topological cyclic homology is the homotopy limit (of the fixed points of topological Hochschild homology) over $\Ar_{\mathcal{M}_n}$ for the monoid $\mathcal{M}_n$ of isogenies.  Topological Frobenius homology is the homotopy limit over the subcategory with $m_2=\id$, and similarly topological restriction homology for the subcategory with $m_4=\id$.  For any submonoid $\cK \subset \cM_n$, there is a construction intermediate between topological cyclic and topological Frobenius homology.  Let $\Ar_{\cM_n}[\cK]$ denote the subcategory of $\Ar_{\cM_n}$ whose morphisms have $m_2 \in \cK$, and define the $\cK$-relative topological cyclic homology by
$$TC^{(n)}_{\cK}(A) := \holim_{m \in \Ar_{\cM_n}[\cK]} T^{m}(A)$$
As before $T^m(A)$ denotes the fixed points of $T(A)$ by the kernel of the isogeny $m$.  Proposition~\ref{prop-tctf} is a special case of the following result.
\begin{theorem}
The monoid $\cK \subset \cM_n$ acts on $TF^{(n)}(A)$ and there is an equivalence
$$TC^{(n)}_{\cK}(A) \simeq \holim_{\cK} TF^{(n)}(A)$$
\end{theorem}
\begin{proof}
Let $\Psi$ denote the functor from $\Ar_{\cM_n}$ to spectra with $\Psi(m) = T^{m}(A)$, that is the functor whose homotopy limit gives topological cyclic homology.  We henceforth abbreviate $\cM_n$ as $\cM$.  We begin by reformulating the relative arrow category as a categorical semi-direct product:
\begin{lemma}
There is an isomorphism
$\Ar_{\cM}[\cK] \cong \cK \ltimes Frob_{\cM}.$
\end{lemma}

\nid Before describing the isomorphism we recall the definition of this semi-direct product.  Given a category $\cC$ and a functor $F: \cC^{\op} \ra \Cat$, the semi-direct product $\cC \ltimes F$ has objects pairs $(c,x)$, where $c \in \cC$ and $x \in F(c)$.  A morphism of $\cC \ltimes F$ from $(c,x)$ to $(c',x')$ is a pair $(f: c \ra c', \phi: x \ra F(f)(x'))$.  The composition of $(f,\phi):(c,x) \ra (c', x')$ and $(g,\psi):(c',x') \ra (c'',x'')$ is $(gf, F(f)(\psi) \circ \phi)$.  This categorical semi-direct product is a slight variant of Thomason's Grothendieck construction~\cite{thomason}.

In the lemma, $Frob_{\cM}$ is viewed as a functor from $\cK^{\op}$ to $\Cat$ taking the single object $\mu$ to $Frob_{\cM}$ and taking the morphism $k$ to the right action functor $k: Frob_{\cM} \ra Frob_{\cM}$ given on objects by $k \cdot (\mu \xra{m} \mu) = \mu \xra{m k} \mu$, and on morphisms by sending $(\id^*,{m_4}_*): m_1 \ra m_3$ to $(\id^*,{m_4}_*): m_1 k \ra m_3 k$.  The isomorphism in the lemma is given by the functor $\cK \ltimes Frob_{\cM} \ra \Ar_{\cM}[\cK]$ taking the pair $(\mu, \mu \xra{m} \mu)$ to the arrow $\mu \xra{m} \mu$ and taking the morphism $(\mu \xra{k} \mu, (\id^*,{m_4}_*): m_1 \ra m_3 k): (\mu, \mu \xra{m_1} \mu) \ra (\mu, \mu \xra{m_3} \mu)$ to the morphism $(k^*,{m_4}_*): m_1 \ra m_3$.

We now need only decompose the homotopy limit over the semi-direct product:
\begin{lemma}
There is a natural equivalence
$\holim_{\cK \ltimes Frob_{\cM}} \Psi \simeq \holim_{\cK} \left( \holim_{Frob_{\cM}} \Psi \right).$
\end{lemma}
\nid Here the functor $\Psi$ on the right hand side is implicitly restricted to the subcategory $Frob_{\cM}$ of $\cK \ltimes Frob_{\cM}$.  This equivalence has nothing to do with the particular categories in question and occurs for any categorical semi-direct product---the result is usually called the Thomason theorem for homotopy limits and is discussed for example in Chach{\'o}lski-Scherer~\cite{chachscher}.
\end{proof}

We use the theorem to express diagonal cyclic homology in terms of Frobenius homology.
\begin{cor}
There is an equivalence $TC^{\Delta}(A) \simeq \holim_{\Delta} TF^{(n)}(A)$.
\end{cor}
\begin{proof}
The chain of equivalences is the following
$$TC^{\Delta}(A) = \holim_{\Ar_{\Delta}} T^{m}(A) \simeq \holim_{\Delta} (\holim_{Frob_{\Delta}} T^{m}(A)) \simeq \holim_{\Delta} (\holim_{Frob_{\cM}} T^{m}(A)) = \holim_{\Delta} TF^{(n)}(A).$$
The first equivalence follows from the theorem, and the second because $Frob_{\Delta}$ is final in $Frob_{\cM}$.
\end{proof}
\nid Because the monoid $\Delta$ is generated by $p \cdot \id_n$, we can reexpress this homotopy limit as
$$\holim_{\Delta} TF^{(n)}(A) \simeq \holim \left( TF^{(n)}(A) \overset{\phi}{\underset{\id}{\rightrightarrows}} TF^{(n)}(A) \right)$$
where $\phi$ is the action of $p \cdot \id_n$ on $TF^{(n)}(A)$ induced by the restriction operator for that matrix.

Now we specialize to studying the diagonal cyclic homology of the sphere, and show that the operator $\phi$ preserves the rank filtration of $TF^{(n)}(\ess)$.  Recall that the rank filtration on the equivariant sphere spectrum, described in section~\ref{sec-rankfilt}, induces a filtration of topological Frobenius homology, and we denote this filtration by $TF^{(n)}(\ess)[k]$.
\begin{prop}
The restriction operator $\phi$ on $TF^{(n)}(\ess)$ preserves the rank filtration $TF^{(n)}(\ess)[k]$.
\end{prop}
\nid This proposition follows by directly checking that the filtered equivariant sphere spectrum functors $\Phi[k](G) := K(\cC^G[k])$ described in section~\ref{sec-rankfilt}, namely the functors whose homotopy limits are $TF^{(n)}(\ess)[k]$, extend from functors on $Frob_{\cM}$ to functors on $\Ar_{\cM}[\Delta]$.

Because the restriction respects the filtration, we have an induced filtration, denoted $TC^{\Delta}(\ess)[k]$, on diagonal cyclic homology.  The filtration quotients here can be described as a homotopy equalizer on the filtration quotients of Frobenius homology:
$$TC^{\Delta}(\ess)[k]/TC^{\Delta}(\ess)[k-1] \simeq \holim \left(
TF^{(n)}(\ess)[k]/TF^{(n)}(\ess)[k-1] \overset{\phi}{\underset{\id}{\rightrightarrows}} TF^{(n)}(\ess)[k]/TF^{(n)}(\ess)[k-1] \right)$$
It remains only to determine the action of $\phi$ on the filtration quotients of Frobenius homology, and for this we utilize the cotype decomposition of section~\ref{sec-cotype}.  Recall the decomposition
$$TF^{(n)}(\ess)[k]/TF^{(n)}(\ess)[k-1] \simeq \prod_{\alpha} \holim_l \Phi[k,\alpha]$$
where $\Phi[k,\alpha]$ is the rank $k$, cotype $\alpha$ part of the equivariant sphere spectrum functor, and where $\alpha$ varies over collections of unordered positive integers $\{e_1, e_2, \ldots, e_k\}$.  Let $P_k$ denote the set of such collections, and define a map $\theta: P_k \ra P_k$ by $\theta(\{e_1, e_2, \ldots, e_k\}) = \{e_1 + 1, e_2 + 1, \ldots, e_k +1\}$.  Roughly speaking, the operator $\phi$ shifts the factors in the cotype decomposition along the map $``\theta^{-1}"$:
\begin{prop}
The operator
$\phi: \prod_{\alpha} \holim_l \Phi[k,\alpha] \ra \prod_{\alpha} \holim_l \Phi[k,\alpha]$
maps the cotype factor $\holim_l \Phi[k,\theta(\alpha)]$ by a homotopy equivalence to the cotype factor $\holim_l \Phi[k,\alpha]$, and for any $\alpha_0$ containing some $e_i=1$, the operator $\phi$ maps the factor $\holim_l \Phi[k,\alpha_0]$ to the basepoint.  

More specifically, there is an identification
$$\prod_{\alpha \in P_k} \holim_l \Phi[k,\alpha] \simeq \prod_{\alpha_0 \in \cO_k} F((\{\theta^j(\alpha_0)\}_{j \geq 0})_+, \holim_l \Phi[k,\alpha_0])$$
where $\cO_k$ is the set of collections $\{e_1, e_2, \ldots, e_k\}$ with at least one $e_i$ equal to 1, and the right hand side of this equivalence is a product of function spectra with discrete sources.  With respect to this identification, the operator $\phi$ is equal to $\prod_{\alpha_0 \in \cO_k} F(\theta, \holim_l \Phi[k,\alpha_0])$.
\end{prop}
\begin{proof}
The identification is seen as follows.  For $\alpha = \{e_1, e_2, \ldots, e_k\}$, choose an \mbox{$l > \max_i (e_i +1)$}.  The restriction associated to the inclusion $C^{\times n}_{p^l} \hra C^{\times n}_{p^{l+1}}$ induces a bijection from \mbox{$\cM[k,\theta(\alpha)](l+1)$} to $\cM[k,\alpha](l)$, where, as in section~\ref{sec-rankcotype}, the expression $\cM[k,\alpha](l)$ denotes the result of adjoining a disjoint basepoint to the set of subgroups $K \subseteq C^{\times n}_{p^l}$ with $C^{\times n}_{p^l}/K$ of rank $k$ and $K$ of cotype $\alpha$.  This bijection yields a distinguished equivalence of $\holim_l \Phi[k,\theta(\alpha)]$ and $\holim_l \Phi[k,\alpha]$ and the identification in the proposition follows.  That the operator $\phi$ corresponds to $F(\theta, -)$ in this identification is seen by directly tracing its action through the given equivalences.
\end{proof}
\nid The homotopy equalizer of $\phi$ and the identity evidently has a single factor for each primitive cotype $\alpha = \{e_1, e_2, \ldots, e_k\}$, that is for those cotypes with some $e_i=1$, and the computation of the rank-cotype components from section~\ref{sec-rankcotype} yields our desired result:
\begin{cor}
The homotopy groups of the filtration quotients of the rank filtration of the diagonal cyclic homology of the sphere are as follows:
$$\pi_*((TC^{\Delta}(\ess)[k]/TC^{\Delta}(\ess)[k-1])_p) \cong \prod_{\alpha \in \mathcal{O}_k} 
\lim_l \left( \bZ[GL_n(\zp) / \Gamma_{l,k,\alpha}] \otimes \pi_*(\Sigma^{\infty} S^k \sm B \TT^k_+)/p^l \right).$$
\end{cor}

\settocdepth{part}

\section*{Acknowledgments}

We would like to express our gratitude to the Institut Mittag-Leffler and to Stanford University for their support and hospitality at various stages of this project.  We would also like to thank Teena Gerhardt and Andre Henriques for helpful background conversations, and Michael Weiss for clarifications concerning our examples of Adams operations on covering homology.

\settocdepth{section}

\bibliography{cdd}

\bibliographystyle{plain}

\end{document}